\documentclass{amsart}

\usepackage{amssymb}
\usepackage{enumitem}
\usepackage{graphicx}

\newtheorem{theorem}{Theorem}[section]
\newtheorem{corollary}[theorem]{Corollary}
\newtheorem{lemma}[theorem]{Lemma}

\newtheorem{problem}[theorem]{Problem}

\theoremstyle{definition}
\newtheorem{definition}[theorem]{Definition}
\newtheorem{remark}[theorem]{Remark}
\newtheorem{example}[theorem]{Example}

\usepackage{color}

\numberwithin{equation}{section}

\begin{document}

\thispagestyle{empty}

\title[Weak Riesz energy. Condensers with touching plates]{A concept of weak Riesz energy with application to condensers with touching plates}

\author{Natalia Zorii}
\address{\footnotesize{Institute of Mathematics\\ Academy of Sciences
of Ukraine\\
Teresh\-chen\-kivska~3\\
01601 Kyiv, Ukraine}}
\email{\footnotesize{natalia.zorii@gmail.com}}


\date{}

\begin{abstract} We proceed further with the study of minimum weak Riesz energy problems for condensers with touching plates, initiated jointly with Bent Fuglede (Potential Anal. {\bf 51} (2019), 197--217). Having now added to the analysis constraint and external source of energy, we obtain a Gauss type problem, but with weak energy involved. We establish sufficient and/or necessary conditions for the existence of solutions to the problem and describe their potentials. Treating the solution as a function of the condenser and the constraint, we prove its continuity relative to the vague topology and the topologies determined by the weak and standard energy norms. We show that the criteria for the solvability thus obtained fail in general once the problem is reformulated in the setting of standard energy, thereby justifying an advantage of weak energy when dealing with condensers with touching plates.
\end{abstract}



\maketitle

\section{Standard and weak Riesz energies of measures}\label{sec:intr} In potential theory on $\mathbb R^n$, $n\geqslant3$, relative to the {\it Riesz kernel\/} $\kappa_\alpha(x,y):=|x-y|^{\alpha-n}$ of order $\alpha\in(0,2]$, we proceed further with the study of minimum weak energy problems, initiated jointly with Bent Fuglede \cite{FZ-Pot1}. We are motivated by the observation that the standard concept of energy is too restrictive when dealing with condensers with touching plates, while application of weak energy allows the treatment of much broader condenser problems.

Let $\mathfrak M(\mathbb R^n)$ stand for the linear space of all real-val\-ued Radon measures $\mu$ on $\mathbb R^n$ equipped with the
{\it vague\/} topology, i.e.\ the topology of pointwise convergence on the class $C_0(\mathbb R^n)$ of all
(real-val\-ued finite) continuous functions on $\mathbb R^n$ with compact support, and let $\mathfrak M^+(\mathbb R^n)$ be the cone of all positive $\mu\in\mathfrak M(\mathbb R^n)$.

The {\it standard\/} concept of $\alpha$-Riesz {\it energy\/} of $\mu\in\mathfrak M(\mathbb R^n)$ is introduced by
\begin{equation}\label{en-st}E_\alpha(\mu):=E_{\kappa_\alpha}(\mu):=\int\kappa_\alpha(x,y)\,d(\mu\otimes\mu)(x,y)\end{equation}
provided that $E_\alpha(\mu^+)+E_\alpha(\mu^-)$ or $E_\alpha(\mu^+,\mu^-)$ is finite, and the finiteness of
$E_\alpha(\mu)$ means that $\kappa_\alpha$ is $(|\mu|\otimes|\mu|)$-integrable, i.e.\ $E_\alpha(|\mu|)<\infty$. Here $\mu^+$ and $\mu^-$ denote the positive and negative parts in the Hahn--Jor\-dan
decomposition of $\mu$,
\[E_\alpha(\mu^+,\mu^-):=\int\kappa_\alpha(x,y)\,d(\mu^+\otimes\mu^-)(x,y)\]
is the (standard) $\alpha$-Riesz {\it mutual energy\/} of $\mu^+$ and $\mu^-$, and $|\mu|:=\mu^++\mu^-$.

The Riesz kernel is {\it strictly positive definite\/} in the sense that $E_\alpha(\mu)\geqslant0$ for any $\mu\in\mathfrak M(\mathbb R^n)$ (whenever defined), and $E_\alpha(\mu)=0$ only for $\mu=0$. This implies that all $\mu\in\mathfrak M(\mathbb R^n)$ with finite $E_\alpha(\mu)$ form a pre-Hilbert space $\mathcal E_\alpha=\mathcal E_\alpha(\mathbb R^n)$ with the (standard) inner product $\langle\mu,\nu\rangle_\alpha:=E_\alpha(\mu,\nu)$ and the (standard energy) norm $\|\mu\|_\alpha:=\sqrt{E_\alpha(\mu)}$. The topology on $\mathcal E_\alpha$ determined by the norm $\|\cdot\|_\alpha$ is said to be {\it strong}. The cone $\mathcal E^+_\alpha:=\mathcal E_\alpha\cap\mathfrak M^+(\mathbb R^n)$ is strongly complete \cite{Ca1}, and this fundamental fact is crucial to the treatment of minimum energy problems over $\mu\in\mathcal E^+_\alpha$.

The whole space $\mathcal E_\alpha$ is however strongly incomplete \cite{Ca1}, which causes substantial difficulties in the investigation of minimum Riesz energy problems for condensers. As was shown earlier by the author (see below for quoted results), those difficulties can be overcome in the framework of the standard approach to definition of energy, provided that a condenser in question satisfies the separation condition (\ref{dist}).

More precisely, fix an (open connected) domain $D\subset\mathbb R^n$ and a relatively closed subset $A$ of $D$. We call the ordered pair $\mathbf A:=(A,F)$, where $F:=D^c:=\mathbb R^n\setminus D$, a ({\it generalized\/}) {\it condenser}, and $A$ and $F$ its {\it plates}. To avoid trivialities, assume
\begin{equation}\label{nonzero}c_\alpha(A)\cdot c_\alpha(F)>0,\end{equation}
where $c_\alpha(\cdot)$ denotes the inner $\alpha$-Riesz {\it capacity\/} \cite{L}.
Let $\mathfrak M(\mathbf A)$ consist of all $\mu\in\mathfrak M(\mathbb R^n)$ such that $\mu^+$ and $\mu^-$ are carried by $A$ and $F$, respectively.\footnote{$\mu\in\mathfrak M(\mathbb R^n)$ is said to be {\it carried\/} by a set $Q\subset\mathbb R^n$ if $Q$ is $|\mu|$-meas\-ur\-able and $Q^c$ is $|\mu|$-neg\-lig\-ible.}
Write
\begin{align}
&\mathfrak M(\mathbf A,\mathbf1):=\bigl\{\mu\in\mathfrak M(\mathbf A): \ \mu^+(A)=\mu^-(F)=1\bigr\},\notag\\
&\mathcal E_\alpha(\mathbf A):=\mathcal E_\alpha\cap\mathfrak M(\mathbf A),\notag\\
&\mathcal E_\alpha(\mathbf A,\mathbf1):=\mathcal E_\alpha\cap\mathfrak M(\mathbf A,\mathbf1),\label{eprcl}
\end{align}
where $\mathbf1:=(1,1)$. In view of (\ref{nonzero}), $\mathcal E_\alpha(\mathbf A,\mathbf1)$ is nonempty (cf.\ \cite[Lemma~2.3.1]{F1}), and hence the following minimum (standard) Riesz energy problem makes sense.

\begin{problem}\label{pr2}Does there exist $\lambda_{\mathbf A}\in\mathcal E_\alpha(\mathbf A,\mathbf1)$ with
\begin{equation}\label{epr}E_\alpha(\lambda_{\mathbf A})=w_\alpha(\mathbf A,\mathbf1):=\inf_{\mu\in\mathcal E_\alpha(\mathbf A,\mathbf1)}\,E_\alpha(\mu)?\end{equation}
\end{problem}

Assume for a moment that
\begin{equation}\label{dist}{\rm dist}(A,F):=\inf_{x\in A, \ y\in F}\,|x-y|>0.\end{equation}
Then a complete description of those $\mathbf A$ for which Problem~\ref{pr2} is solvable has been established in \cite[Section~5.1]{Z1}; see Example~\ref{rem-ex} below for an illustration of the results obtained. The approach applied is based on a strong completeness theorem stating that for any $q\in(0,\infty)$, the topological subspace $\mathcal E_\alpha^{\leqslant q}(\mathbf A)$ of $\mathcal E_\alpha$ consisting of all $\mu\in\mathcal E_\alpha(\mathbf A)$ with $|\mu|(\mathbb R^n)\leqslant q$ is complete in the induced strong topology, and the strong topology on $\mathcal E_\alpha^{\leqslant q}(\mathbf A)$ is stronger than the vague topology \cite[Theorem~1]{Z1}.\footnote{Later this approach has been extended to a general function kernel on a locally compact Hausdorff space and vector Radon measures, finite or infinite dimensions \cite{ZPot1,ZPot2,ZPot3}. See e.g.\ \cite[Theorem~13.1]{ZPot1} and \cite[Theorem~9.1]{ZPot2} for a vector analogue of \cite[Theorem~1]{Z1}.\label{foot3}} In particular, it follows from \cite[Section~5.1]{Z1} that if $D^c$ is not $\alpha$-thin at infinity, then a solution $\lambda_{\mathbf A}$ to Problem~\ref{pr2} exists if and only if there is an equilibrium measure $\gamma_A$ on $A$ relative to the $\alpha$-Green kernel on $D$; and moreover these $\lambda_{\mathbf A}$ and $\gamma_A$ are related to one another by the formula
\[\lambda_{\mathbf A}=(\gamma_A-\gamma_A')/\gamma_A(D),\]
where $\gamma_A'$ is the $\alpha$-Riesz balayage of $\gamma_A$ onto $F$. See Section~\ref{sec:prel} below for relevant definitions.

However, for a condenser with touching plates (when ${\rm dist}(A,F)=0$ is allowed), the approach that has been worked out in \cite{Z1} breaks down. Moreover, then the quoted result on the solvability of Problem~\ref{pr2} fails in general since one can construct a domain $D\subset\mathbb R^3$ and a relatively closed set $A\subset D$ such that the (classical) $2$-Green equilibrium measure $\gamma_A$ has infinite Newtonian energy \cite[Example~10.1]{FZ-Pot2}.

As shown in \cite[Theorem~6.1]{FZ-Pot1}, the quoted result on the solvability of Problem~\ref{pr2} nevertheless does hold if the class $\mathcal E_\alpha(\mathbf A,\mathbf1)$ of the admissible measures in (\ref{epr}) is replaced by a properly chosen class of $\mu\in\mathfrak M(\mathbf A,\mathbf1)$ with finite {\it weak\/} $\alpha$-Riesz energy $\dot{E}_\alpha(\mu)$, $\dot{E}_\alpha(\mu)$ being defined essentially (see \cite[Definition~4.1]{FZ-Pot1}) by
\begin{equation}\label{weak}{\dot E}_\alpha(\mu):=\int(\kappa_{\alpha/2}\mu)^2\,dm,\end{equation}
where $\kappa_{\alpha/2}\mu$ is the potential of $\mu$
relative to the $\alpha/2$-Riesz kernel $|x-y|^{\alpha/2-n}$, cf.\ (\ref{p-st}), and $m$ the $n$-dim\-ens\-ional
Lebesgue measure. (See Section~\ref{more} below for some further details about the concept of weak energy.)

Being thus motivated to investigate further condenser problems in the setting of weak Riesz energy, we shall now add to the analysis constraint and external source of energy, thereby obtaining a Gauss type problem, but with weak energy involved. Appearing as a natural generalization of the well-known (standard) constrained Gauss variational problem to measures with finite weak energy, this problem may also be of interest for mathematicians working with orthogonal polynomials and rational approximations.
See Sections~\ref{sec:cond} and \ref{sec-main} below for the strict formulation of the problem in question and the results obtained.

In what follows we shall tacitly use the notions and notation introduced above.

\section{Preliminaries}\label{sec:prel}

In this section we have compiled some basic facts of $\alpha$-Riesz and $\alpha$-Green potential theory that will be used throughout the paper.

When speaking of a positive Radon measure $\mu$ on $\mathbb R^n$, we shall always tacitly assume that its $\alpha$-Riesz {\it potential\/} $\kappa_\alpha\mu$ is not identically infinite:
\begin{equation}\label{p-st}\kappa_\alpha\mu:=\int\kappa_\alpha(\cdot,y)\,d\mu(y)\not\equiv\infty,\end{equation}
or equivalently \cite[Chapter~I, Section~3, n$^\circ$~7]{L}
\[\int_{|y|>1}\,\frac{d\mu(y)}{|y|^{n-\alpha}}<\infty.\]
Then (and only then) the potential $\kappa_\alpha\mu$ of any (signed) $\mu\in\mathfrak M(\mathbb R^n)$ is well-def\-ined and finite {\it nea\-rly everywhere\/} ({\it n.e.}) on $\mathbb R^n$, namely everywhere except for a set of zero inner $\alpha$-Riesz capacity $c_\alpha(\cdot)$; see \cite[Chap\-ter~III, Section~1]{L}.

A measure $\mu\in\mathfrak M(\mathbb R^n)$ is said to be {\it absolutely continuous\/} if $|\mu|(K)=0$ for every compact set $K\subset\mathbb R^n$ with $c_\alpha(K)=0$. Any $\mu\in\mathcal E_\alpha$ is certainly absolutely continuous; but not conversely \cite[pp.~134--135]{L}.

\subsection{Radon measures on $D$. $\alpha$-Riesz balayage}\label{sec-ext} Treating a given domain $D$ (see Section~\ref{sec:intr}) as a locally compact space, we consider the linear space $\mathfrak M(D)$ of all real-val\-ued Radon measures $\nu$ on $D$ equipped with the vague topology of pointwise convergence on the class $C_0(D)$ of all continuous  functions  with compact (in $D$) support, and the cone $\mathfrak M^+(D)$ of all positive $\nu\in\mathfrak M(D)$.

\begin{lemma}[{\rm see e.g.\ \cite[Section~1.1]{F1}}]\label{lemma-semi}If a lower semicontinuous\/ {\rm(}l.s.c.{\rm)} function\/ $\psi:D\to(-\infty,\infty]$ is either positive or of compact\/  {\rm(}in\/ $D${\rm)} support, then the mapping\/ $\nu\mapsto\int\psi\,d\nu$ is
vaguely l.s.c.\ on\/ $\mathfrak M^+(D)$.\end{lemma}

We call $\nu\in\mathfrak M(D)$ {\it extendible\/} by $0$ outside $D$ if there is $\breve{\nu}\in\mathfrak M(\mathbb R^n)$ such that
\[\breve{\nu}^\pm(\varphi)=\int\varphi|_D\,d\nu^\pm\text{ \ for all\ }\varphi\in C_0(\mathbb R^n),\]
and we identify this $\nu$ with its extension $\breve{\nu}$. Any {\it bounded\/} $\nu\in\mathfrak M(D)$, i.e.\ with $|\nu|(D)<\infty$, is extendible; but not the other way around. Also note that the trace (restriction) $\mu|_D$ of any $\mu\in\mathfrak M(\mathbb R^n)$ on the (Borel) set $D$ is certainly extendible.

For any $Q\subset D$ we denote $\breve{\mathfrak M}(Q)$ the linear space of all extendible $\nu\in\mathfrak M(D)$ carried by $Q$,\footnote{If $Q$ is Borel, then $\breve{\mathfrak M}(Q)$ consists in fact of all the restrictions $\mu|_Q$, $\mu$ ranging over $\mathfrak M(\mathbb R^n)$.} and $\breve{\mathfrak M}^+(Q)$ the cone of all positive $\nu\in\breve{\mathfrak M}(Q)$. Write
\[\breve{\mathfrak M}^+(Q,1):=\bigl\{\nu\in\breve{\mathfrak M}^+(Q): \ \nu(Q)=1\bigr\}.\]

In view of these definitions the concepts of $\alpha$-Riesz potential theory can equally well be applied to measures $\nu\in\breve{\mathfrak M}(D)$ (or, to be exact, to their extensions $\breve{\nu}$). In particular, for any $\nu\in\breve{\mathfrak M}^+(D)$ there is a unique absolutely continuous measure $\nu'\in\mathfrak M^+(\mathbb R^n)$ carried by $F$ (${}=D^c$) with
\begin{equation}\label{bal-eq}\kappa_\alpha\nu'=\kappa_\alpha\nu\text{ \ n.e. on\ }F,\end{equation}
see \cite[Corollaries~3.19, 3.20]{FZ}; this $\nu'$ is said to be the $\alpha$-Riesz {\it balayage\/} of $\nu$ onto $F$. According to \cite[Theorem~3.17]{FZ}, the balayage $\nu'$ can be written in the form\footnote{In the literature
the integral representation (\ref{L-repr}) seems to have been more or less taken for granted,
though it has been pointed out in \cite[Chapter~V, Section~3, n$^\circ$~1]{B2} that it requires that the
family $(\varepsilon_y')_{y\in D}$ be $\mu${\it -ad\-equ\-ate\/} in the sense of
\cite[Chapter~V, Section~3, Definition~1]{B2}; see also counterexamples (without $\mu$-ad\-equ\-acy)
in Exercises~1 and~2 at the end of that section. A proof of this adequacy has therefore been given in
\cite[Lemma~3.16]{FZ}. However, {\it the question whether the integral representation holds for any positive Radon measure\/ $\mu$ on\/ $\mathbb R^n$ is still open}.\label{foot}}
\begin{equation}\label{L-repr}\nu'=\int\varepsilon_y'\,d\nu(y),\end{equation}
where $\varepsilon_y$ is the unit Dirac measure at $y\in D$. If moreover $E_\alpha(\nu)<\infty$, then the balayage $\nu'$ is in fact
the orthogonal projection of $\nu$ in the pre-Hilbert space $\mathcal E_\alpha$ onto the convex cone $\mathcal E^+_\alpha(F)$ of all $\mu\in\mathcal E_\alpha^+$ carried by $F$; that is, $\nu'\in\mathcal E^+_\alpha(F)$ and
\begin{equation}\label{proj}\|\nu-\nu'\|_\alpha<\|\nu-\eta\|_\alpha\text{ \ for all\ }
\eta\in\mathcal E^+_\alpha(F), \ \eta\ne\nu'.\end{equation}
The mapping $\nu\mapsto\nu'$, $\nu\in\breve{\mathfrak M}^+(D)$, is extended to signed $\nu\in\breve{\mathfrak M}(D)$ by linearity.

\subsection{$\alpha$-thinness at infinity} According to \cite[Theorem~3.11]{FZ},
\begin{equation}\label{m-est}\nu'(\mathbb R^n)\leqslant\nu(\mathbb R^n)\text{ \ for every\ }\nu\in\breve{\mathfrak M}^+(D).\end{equation}
Before specifying this estimate in Theorem~\ref{bal-mass-th}, we introduce the following notion.

\begin{definition}\label{def-th}We call a closed set $Q\subset\mathbb R^n$ {\it $\alpha$-thin
at infinity\/} if either $Q$ is compact, or $x=0$ is $\alpha$-ir\-reg\-ular for the inverse of $Q$ relative to
$\{x: |x|=1\}$, the $\alpha$-ir\-reg\-ul\-arity being defined e.g.\ by a Wiener type criterion \cite[Theorem~5.2]{L}.
\end{definition}

\begin{theorem}[{\rm see \cite[Theorems~8.6, 8.7]{Z-bal}}]\label{bal-mass-th} $D^c$ is\/
$\alpha$-thin at infinity if and only if there exists a nonzero\/ $\nu\in\breve{\mathfrak M}^+(D)$ with\/\footnote{This result has been announced earlier in \cite{Z1}. The proof provided in \cite{Z1} was however incomplete, being based on the integral representation (\ref{L-repr}); see footnote~\ref{foot} for more details.}
\begin{equation*}\label{m}\nu'(\mathbb R^n)<\nu(\mathbb R^n).\end{equation*}
\end{theorem}

\begin{remark}\label{rem-thin}If a closed set $Q\subset\mathbb R^n$ is not $\alpha$-thin at infinity, then $c_\alpha(Q)=\infty$. Indeed, by the Wiener criterion, $Q$ is not $\alpha$-thin at infinity
if and only if
\[\sum_{k\in\mathbb N}\,\frac{c_\alpha(Q_k)}{q^{k(n-\alpha)}}=\infty,\]
where $q>1$ and $Q_k:=Q\cap\bigl\{x\in\mathbb R^n: q^k\leqslant|x|<q^{k+1}\bigr\}$, while by
\cite[Lemma~5.5]{L} $c_\alpha(Q)<\infty$ is equivalent to the relation
\[\sum_{k\in\mathbb N}\,c_\alpha(Q_k)<\infty.\]
These observations also imply that the converse is not true, i.e.\ there is $Q$ with $c_\alpha(Q)=\infty$,
but $\alpha$-thin at infinity.\end{remark}

\begin{figure}[htbp]
\begin{center}
\vspace{-.4in}
\hspace{1in}\includegraphics[width=4in]{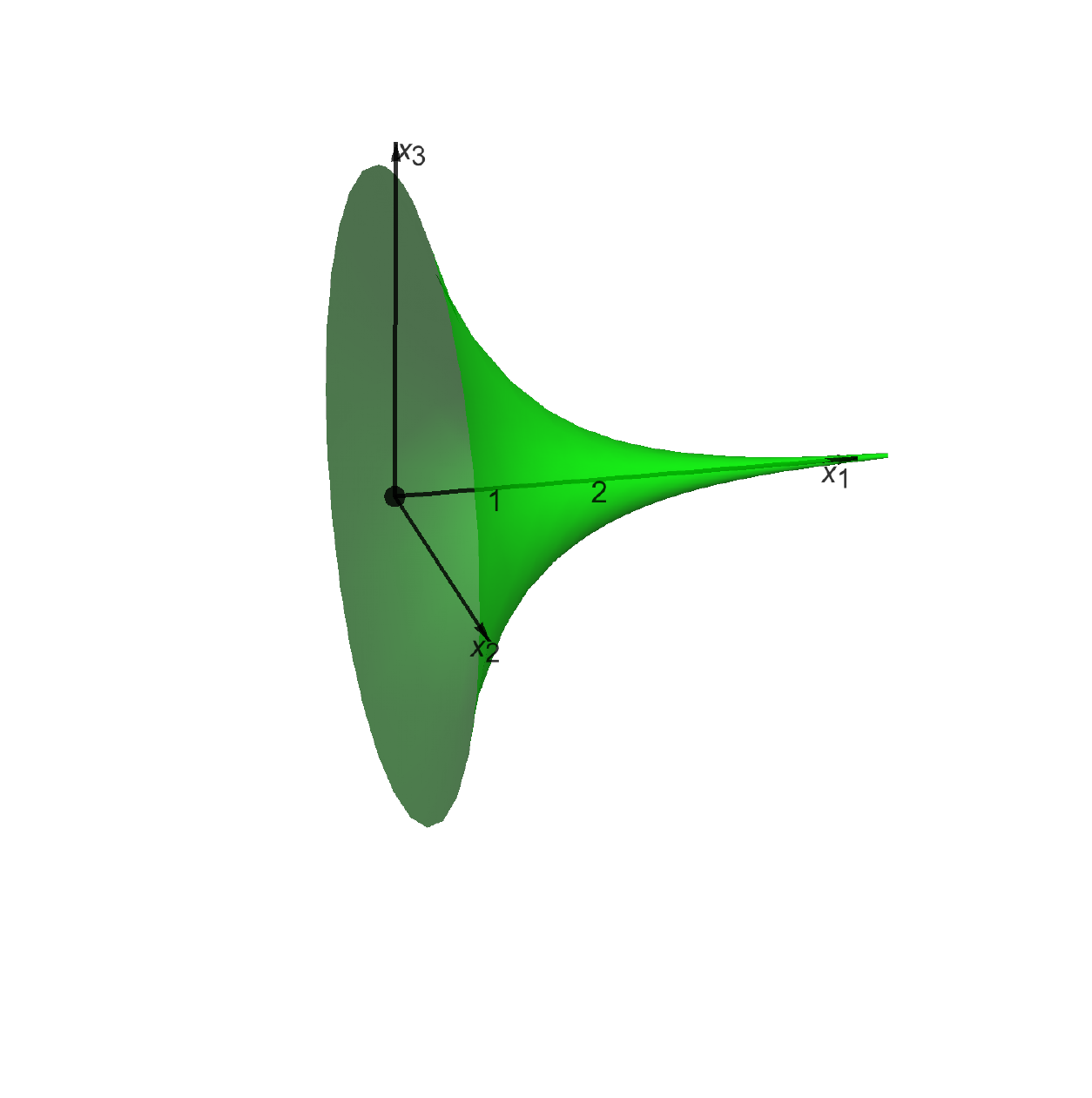}
\vspace{-.8in}
\caption{The set $Q_\varrho$ in Examples~\ref{ex}, \ref{rem-ex} with $\varrho(x_1)=\exp(-x_1)$.\vspace{-.1in}}
\label{Fig1}
\end{center}
\end{figure}

\begin{example}\label{ex} Let $n=3$ and $\alpha=2$. Define the rotation body
\begin{equation}\label{descr}Q_\varrho:=\bigl\{x\in\mathbb R^3: \ 0\leqslant x_1<\infty, \
x_2^2+x_3^2\leqslant\varrho^2(x_1)\bigr\},\end{equation}
where $\varrho$ is given by one of the following three formulae:
\begin{align}
\label{c1}\varrho(x_1)&=x_1^{-s}\text{ \ with\ }s\in[0,\infty),\\
\label{c2}\varrho(x_1)&=\exp(-x_1^s)\text{ \ with\ }s\in(0,1],\\
\label{c3}\varrho(x_1)&=\exp(-x_1^s)\text{ \ with\ }s\in(1,\infty).
\end{align}
Then $Q_\varrho$ is not $2$-thin at infinity if $\varrho$ is defined by (\ref{c1}),
$Q_\varrho$ is $2$-thin at infinity but $c_2(Q_\varrho)=\infty$ if $\varrho$ is given
by (\ref{c2}) (see Figure~\ref{Fig1}),
and $c_2(Q_\varrho)<\infty$ if (\ref{c3}) holds.
\end{example}

\begin{remark}\label{rem-comp}For $\alpha=2$, the concept of $\alpha$-thinness at infinity thus defined is, in fact, equivalent to that by Doob \cite[pp.~175--176]{Doob}, while for $\alpha\ne2$, it seems to appear first in our earlier work \cite{Z1}. Due to its deep relation with balayage, observed in Theorem~\ref{bal-mass-th}, it plays an important role in the investigation of condenser problems in Riesz potential theory (see e.g.\ \cite{DFHSZ2,FZ-Pot1,FZ-Pot2}; for an illustration, see Example~\ref{rem-ex} below). Note that for $\alpha=2$, a different concept of $\alpha$-thin\-ness at infinity has been introduced by Brelot \cite[p.~313]{Brelot}, which is actually more restrictive than Doob's (equivalently, our) concept. Indeed, a closed set $Q\subset\mathbb R^n$ is $2$-thin at infinity by Brelot if and only if $c_2(Q)<\infty$ (see \cite[p.~277, footnote]{Ca2} or \cite[Chapter~IX, Section~6]{Brelo2}); while according to Remark~\ref{rem-thin}, $c_\alpha(Q)<\infty$ is only sufficient, but not necessary for $Q$ to be $\alpha$-thin at infinity in the sense of our Definition~\ref{def-th}.\end{remark}

Throughout the rest of the paper, when speaking of $\alpha$-thin\-ness at infinity, we shall tacitly follow Definition~\ref{def-th}. Then a closed set $Q\subset\mathbb R^n$ is, in fact, $\alpha$-thin at infinity if and only if there is an $\alpha$-Riesz equilibrium measure on $Q$, treated in an extended sense where infinite energy is allowed (see \cite[Section~5]{Z-bal}).

\subsection{$\alpha$-Green kernel}\label{sec-gr}Write
\begin{equation}\label{gr}g(x,y):=g^\alpha_D(x,y):=\kappa_\alpha\varepsilon_y(x)-\kappa_\alpha\varepsilon_y'(x)\text{ \ for all\ }x,y\in D.\end{equation}
Being positive and l.s.c.\ on $D\times D$ (see e.g.\ \cite{DFHSZ}), the function $g$ can serve as a kernel of potential theory on the locally compact space $D$.
We define the energy $E_g(\nu)$ and the potential $g\nu$ of $\nu\in\mathfrak M(D)$ relative to the {\it $\alpha$-Green kernel\/} $g$ by replacing  $\kappa_\alpha$ by $g$ in (\ref{en-st}) and (\ref{p-st}), respectively. As the kernel $g$ is strictly positive definite \cite[Theorem~4.9]{FZ} (for $\alpha=2$, see also \cite[Chapter~XIII, Section~7]{Doob}), all $\nu\in\mathfrak M(D)$ with $E_g(\nu)<\infty$ form a pre-Hil\-bert space $\mathcal E_g=\mathcal E_g(D)$ with the inner product
\[\langle\mu,\nu\rangle_g:=E_g(\mu,\nu):=\int g\mu\,d\nu\text{ \ for all\ }\mu,\nu\in\mathcal E_g.\]
The topology on $\mathcal E_g$ defined by the norm $\|\nu\|_g:=\sqrt{E_g(\nu)}$ is said to be {\it strong}.

Applying \cite[Lemma~2.3.1]{F1} shows that any Borel set $B\subset D$ with zero inner $\alpha$-Green capacity $c_g(\cdot)$ {\it cannot carry\/} any $\nu\in\mathcal E_g$, i.e.\ $|\nu|(B)=0$. More generally, this holds for any $\nu\in\mathfrak M(D)$ with $\nu|_K\in\mathcal E_g$ for every compact $K\subset D$. Also note that for any $Q\subset D$, $c_g(Q)=0\iff c_\alpha(Q)=0$, see e.g.\ \cite[Lemma~2.6]{DFHSZ}.

Fix $\nu\in\breve{\mathfrak M}(D)$. Integrating (\ref{gr}) with respect to $\nu$ and applying (\ref{L-repr}) to $\nu^\pm$, we obtain by \cite[Chapter~V, Section~3, Proposition~1]{B2}
\begin{equation}\label{3.4}g\nu=\kappa_\alpha\nu-\kappa_\alpha\nu'\text{ \ n.e.\ on\ }D,\end{equation}
both $\kappa_\alpha\nu$ and $\kappa_\alpha\nu'$ being finite n.e.\ on $\mathbb R^n$ in consequence of our general agreement. As for relations between $\alpha$-Green, standard $\alpha$-Riesz, and weak $\alpha$-Riesz energies, first note that $E_\alpha(\nu)<\infty$ implies $E_g(\nu)<\infty$. For positive $\nu$, this is obvious since
\[g(x,y)<|x-y|^{\alpha-n}\text{ \ for all\ }x,y\in D,\]
while for signed $\nu$, this follows from the definition of finite (standard) energy.

\begin{lemma}\label{eq-r-g}Assume that\/ $\nu\in\breve{\mathfrak M}(D)$ is bounded and
\begin{equation}\label{dist'}\inf_{x\in S_D^\nu, \ z\in F}\,|x-z|>0,\end{equation}
$S_D^\nu$ being the support of\/ $\nu$ in\/ $D$. Then\/ $E_g(\nu)$ is finite\/ {\rm(}if and\/{\rm)} only if\/ $E_\alpha(\nu)$ is so.
\end{lemma}

\begin{proof} It is enough to verify this for $\nu\geqslant0$. In view of (\ref{m-est}), we
get from (\ref{dist'})
\[\kappa_\alpha\varepsilon_y'(x)=\int|x-z|^{\alpha-n}\,d\varepsilon_y'(z)\leqslant C\text{ \ for all\ }x,y\in S_D^\nu,\]
where a constant $C\in(0,\infty)$ is independent of $x,y$. Hence
\[\kappa_\alpha(x,y)=g(x,y)+\kappa_\alpha\varepsilon_y'(x)\leqslant g(x,y)+C\text{ \ for all\ }x,y\in S_D^\nu,\]
and therefore
\begin{equation}\label{10}E_\alpha(\nu)\leqslant E_g(\nu)+C\nu(D)^2.\end{equation}
Thus $E_g(\nu)<\infty$ implies, indeed, $E_\alpha(\nu)<\infty$.
\end{proof}

\begin{remark}\label{2.6}If assumption (\ref{dist'}) is dropped, then Lemma~\ref{eq-r-g} fails in general. See \cite[Example~10.1]{DFHSZ2} for an example of a domain $D$ in $\mathbb R^3$ and a bounded $\nu_0\in\mathfrak M^+(D)$ with $E_{g^2_D}(\nu_0)<\infty$, but $E_2(\nu_0)=\infty$. Compare with Theorem~\ref{thm} below.\end{remark}

\begin{theorem}[{\rm see \cite[Theorem~5.1]{FZ-Pot1}}]\label{thm} For any\/ $\nu\in\breve{\mathfrak M}(D)$ with\/ $E_g(\nu)<\infty$, $\nu-\nu'$ has finite
weak\/ $\alpha$-Riesz energy\/ $\dot{E}_\alpha(\nu-\nu')$, defined by\/ {\rm(\ref{weak})}, and moreover
\begin{equation}\label{2.8} E_g(\nu)=\dot{E}_\alpha(\nu-\nu').
\end{equation}
\end{theorem}

\begin{corollary}\label{cor}For any\/ $\nu\in\breve{\mathfrak M}(D)$ with\/ $E_\alpha(\nu)<\infty$,
\[E_g(\nu)=E_\alpha(\nu-\nu').\]
\end{corollary}

\begin{proof}Indeed, then both $E_g(\nu)$ and $E_\alpha(\nu-\nu')$ are finite, hence
\[E_g(\nu)=\dot{E}_\alpha(\nu-\nu')=E_\alpha(\nu-\nu'),\]
the former equality being valid by (\ref{2.8}) and the latter by (\ref{dot0}) (see below).\end{proof}

Crucial to our study is the {\it perfectness\/} of the kernel $g$, established recently in \cite[Theorem~4.11]{FZ}, which amounts to the completeness of the cone $\mathcal E_g^+$ of all positive $\nu\in\mathcal E_g$ in the (induced) strong topology. In more detail, every strong Cauchy sequence (net) in $\mathcal E_g^+$ converges strongly to any of its vague cluster points, and the strong topology on $\mathcal E_g^+$ is stronger than the (induced) vague topology on $\mathcal E_g^+$. Combining this with \cite[Theorem~4.1]{F1} results in the following fundamental fact.

\begin{theorem}\label{th-equi} For any\/ $Q\subset D$ with\/ $c_g(Q)<\infty$, there is a unique\/ $\gamma=\gamma_Q\in\mathcal E^+_g$ with the properties\/ $S_D^\gamma\subset\text{\rm Cl}_DQ$ and\/\footnote{The kernel $g$ satisfies the complete maximum principle in the form stated in \cite[Theorem~4.6]{FZ}, hence the Frostman maximum principle. Therefore, (\ref{eq2}) implies $g\gamma_Q\leqslant1$ on $D$.}
\begin{align}\label{eq1}&E_g(\gamma)=\gamma(D)=c_g(Q),\\
\label{eq2}&g\gamma=1\text{ \ n.e.\ on\ }Q.
\end{align}
This\/ $\gamma$, termed the\/ $g$-equilibrium measure for\/ $Q$, solves the problem of minimizing\/ $E_g(\nu)$ over the class\/ $\Gamma_Q$ of all\/ {\rm(}signed\/{\rm)} $\nu\in\mathcal E_g$ with\/ $g\nu\geqslant1$ n.e.\ on\/ $Q$, that is,
\[c_g(Q)=\|\gamma\|^2_g=\min_{\nu\in\Gamma_Q}\,\|\nu\|^2_g.\]
\end{theorem}

\begin{remark}\label{qe}Assume $Q$ is Borel. Then one could equally well write `{\it q.e.}' ({\it quasi everywhere\/}) instead of `n.e.' in (\ref{eq2}), where `q.e.' refers to {\it outer\/}  $\alpha$-Riesz capacity \cite[Chapter~II, Section~2, n$^\circ$~6]{L}. Indeed, being a l.s.c.\ function, $g\gamma$ is Borel measurable, and hence $\{x\in Q:\ g\gamma(x)\ne1\}$ is {\it cap\-acit\-able\/} \cite[Theorem~2.8]{L}.\end{remark}

\subsection{More about weak Riesz energy}\label{more}As shown in \cite[Section~4]{FZ-Pot1}, all $\mu\in\mathfrak M(\mathbb R^n)$ with finite weak $\alpha$-Riesz energy
${\dot E}_\alpha(\mu)$, or equivalently with
\[\kappa_{\alpha/2}\mu\in L^2(m),\] form a pre-Hil\-bert space $\dot{\mathcal E}_\alpha=\dot{\mathcal E}_\alpha(\mathbb R^n)$ with
the (weak) inner product
\[\langle\mu,\nu\rangle^\cdot_\alpha:=
\langle\kappa_{\alpha/2}\mu,\kappa_{\alpha/2}\nu\rangle_{L^2(m)}:=\int\kappa_{\alpha/2}\mu\cdot\kappa_{\alpha/2}\nu\,dm\]
and the (weak energy) norm
\[\|\mu\|^{\cdot}_\alpha:=\sqrt{{\dot E}_\alpha(\mu)}=\|\kappa_{\alpha/2}\mu\|_{L^2(m)}.\]
The Riesz composition identity $\kappa_\alpha=\kappa_{\alpha/2}\ast\kappa_{\alpha/2}$, see \cite[Section~1, Eq.~(12)]{Riesz} or \cite[Eq.~(1.1.12)]{L}, implies that
\begin{equation}\label{w-incl}\mathcal E^+_\alpha=\dot{\mathcal E}^+_\alpha\text{ \ and \ }
\mathcal E_\alpha\subset\dot{\mathcal E}_\alpha,\end{equation}
where $\dot{\mathcal E}^+_\alpha:=\dot{\mathcal E}_\alpha\cap\mathfrak M^+(\mathbb R^n)$; and moreover
\begin{equation}\label{dot0}\|\mu\|_\alpha=\|\mu\|_\alpha^\cdot\text{ \ and \ }\langle\mu,\nu\rangle_\alpha=\langle\mu,\nu\rangle^\cdot_\alpha\text{ \ for all\ }\mu,\nu\in\mathcal E_\alpha.\end{equation}

We emphasize that $\mathcal E_\alpha$ is a {\it proper\/} subset of $\dot{\mathcal E}_\alpha$. Indeed, take $D\subset\mathbb R^3$ and $\nu_0$ as in Remark~\ref{2.6} above, and write $\mu_0:=\nu_0-\nu_0'$. Then $\dot{E}_2(\mu_0)=E_{g^2_D}(\nu_0)<\infty$ according to Theorem~\ref{thm}, hence $\mu_0\in\dot{\mathcal E}_2$, but $\mu_0\not\in\mathcal E_2$ because $E_2(\nu_0)=\infty$.

\begin{remark}\label{rem-more}An example similar to that by Cartan \cite{Ca1} shows that $\dot{\mathcal E}_\alpha$ is incomplete in the topology determined by the weak energy norm $\|\cdot\|^{\cdot}_\alpha$. The following two facts are also worth mentioning, though not being used in this study.
\begin{itemize}
\item[1.] $\mathcal E_\alpha$ is {\it dense\/} in $\dot{\mathcal E}_\alpha$ in both the vague topology and the topology determined by the weak energy norm, see \cite[Theorem~4.1]{FZ-Pot1}.
    \item[2.] The pre-Hil\-bert space $\dot{\mathcal E}_\alpha$ is isometrically imbedded into its completion, the Hilbert space of all tempered distributions on
$\mathbb R^n$ with finite Deny--Schwartz energy, defined with the aid of Fourier transform. See \cite[Theorem~4.2]{FZ-Pot1} for details; this result for $\mathcal E_\alpha$ in place of $\dot{\mathcal E}_\alpha$ goes back to Deny \cite{D1}.\end{itemize}\end{remark}

\section{Weak constrained Gauss variational problem}\label{sec:cond}

Fix a (generalized) condenser $\mathbf A=(A,F)$; for the definition and permanent assumptions, see Section~\ref{sec:intr}. Parallel with the class $\mathcal E_\alpha(\mathbf A,\mathbf1)$ given by (\ref{eprcl}), consider
\[\dot{\mathcal E}_\alpha(\mathbf A,\mathbf1):=\dot{\mathcal E}_\alpha\cap\mathfrak M(\mathbf A,\mathbf1).\]
In this paper we shall be interested in minimum weak $\alpha$-Riesz energy problems over certain subclasses of $\dot{\mathcal E}_\alpha(\mathbf A,\mathbf1)$. More precisely,
we treat any $\sigma\in\mathfrak M^+(D)$ carried by $A$
with $\sigma(A)>1$ as an (upper) {\it constraint\/} acting on measures of the class $\breve{\mathfrak M}^+(A,1)$, and we denote by $\mathfrak C(A)$ the collection of all those $\sigma$. For any $\sigma\in\mathfrak C(A)$, write
\[\breve{\mathfrak M}^\sigma(A,1):=\bigl\{\nu\in\breve{\mathfrak M}^+(A,1): \ \nu\leqslant\sigma\bigr\},\]
where $\nu\leqslant\sigma$ means that $\sigma-\nu\in\mathfrak M^+(D)$.
We use the formal notation $\sigma=\infty$ to indicate that there is {\it no\/} upper constraint, that is,
\[\breve{\mathfrak M}^\infty(A,1):=\breve{\mathfrak M}^+(A,1).\]

Fix $\sigma\in\mathfrak C(A)\cup\{\infty\}$, and write
\begin{align}&\mathcal E_\alpha^\sigma(\mathbf A,\mathbf1):=\bigl\{\mu\in\mathcal E_\alpha(\mathbf A,\mathbf1): \ \mu^+\in\breve{\mathfrak M}^\sigma(A,1)\bigr\},\label{stG}\\
&\dot{\mathcal E}_\alpha^\sigma(\mathbf A,\mathbf1):=\bigl\{\mu\in\dot{\mathcal E}_\alpha(\mathbf A,\mathbf1): \ \mu^+\in\breve{\mathfrak M}^\sigma(A,1)\bigr\},\notag\\
&\ddot{\mathcal E}_\alpha^\sigma(\mathbf A,\mathbf1):=\dot{\mathcal E}_\alpha^\sigma(\mathbf A,\mathbf1)\bigcap\Bigl\{\bigcap_{k\in\mathbb N}\text{\rm Cl}_{\dot{\mathcal E}_\alpha}\,\mathcal E_\alpha^{(1+\frac1k)\sigma}(\mathbf A,\mathbf1)\Bigr\}.\label{def_E2}
\end{align}
Equivalently, $\ddot{\mathcal E}_\alpha^\sigma(\mathbf A,\mathbf1)$ consists of all $\mu\in\dot{\mathcal E}_\alpha^\sigma(\mathbf A,\mathbf1)$ such that for every $\varepsilon>0$ there is $\mu_\varepsilon\in\mathcal E_\alpha(\mathbf A,\mathbf1)$ with the properties
\[\|\mu_\varepsilon-\mu\|_\alpha^\cdot<\varepsilon\text{ \ and \ }\mu_\varepsilon^+\leqslant(1+\varepsilon)\sigma.\]
In the {\it unconstrained\/} case ($\sigma=\infty$) the latter estimate can be dropped, and hence $\ddot{\mathcal E}_\alpha^\infty(\mathbf A,\mathbf1)$ reduces to the class\footnote{Although $\mathcal E_\alpha$ is dense in $\dot{\mathcal E}_\alpha$ in the vague topology and the topology determined by the weak energy norm, {\it the question if\/ $\mathcal E_\alpha(\mathbf A,\mathbf1)$ is dense in\/ $\dot{\mathcal E}_\alpha(\mathbf A,\mathbf1)$ is still open}.}
\begin{equation}\label{circ}\ddot{\mathcal E}_\alpha(\mathbf A,\mathbf1):=\ddot{\mathcal E}_\alpha^\infty(\mathbf A,\mathbf1)=\dot{\mathcal E}_\alpha(\mathbf A,\mathbf1)\cap\text{\rm Cl}_{\dot{\mathcal E}_\alpha}\,\mathcal E_\alpha(\mathbf A,\mathbf1).\end{equation}

To ensure that $\mathcal E_\alpha^\sigma(\mathbf A,\mathbf1)$, hence $\ddot{\mathcal E}_\alpha^\sigma(\mathbf A,\mathbf1)$, is {\it nonempty}, in what follows we shall tacitly assume that either $\sigma=\infty$ holds, or  $\sigma\in\mathfrak C(A)$ has the property\footnote{Then for any compact set $K_0\subset D$ with $\sigma(K_0)>1$, we have $\sigma|_{K_0}/\sigma(K_0)\in\mathcal E_\alpha\cap\breve{\mathfrak M}^\sigma(A,1)$, and hence $\mathcal E_\alpha^\sigma(\mathbf A,\mathbf1)$ is indeed nonempty.}
\begin{equation}\label{adm-c}\sigma|_K\in\mathcal E^+_\alpha\text{ \ for every compact\ }K\subset D\end{equation}
(equivalently, $\sigma|_K\in\mathcal E^+_g$ for every compact $K\subset D$, cf.\ Lemma~\ref{eq-r-g}).

Fix also a (signed) Radon measure $\vartheta$ on $\mathbb R^n$ carried by $D$ with $E_\alpha(\vartheta)<\infty$, i.e.
\begin{equation}\label{g}\vartheta\in\mathcal E_\alpha\cap\breve{\mathfrak M}(D);\end{equation}
the measure (charge) $\vartheta-\vartheta'$, where $\vartheta'$ is the $\alpha$-Riesz balayage of $\vartheta$ onto $F$, will be thought of as an {\it external source of energy}. Having observed that
\[\vartheta-\vartheta'\in\mathcal E_\alpha\subset\dot{\mathcal E}_\alpha,\]
we define the {\it weak Gauss integral\/} $\dot{G}_{\alpha,\vartheta-\vartheta'}(\mu)$ for $\mu\in\dot{\mathcal E}_\alpha$ by
\begin{align}\label{defG2}\dot{G}_{\alpha,\vartheta-\vartheta'}(\mu)&:=\|\mu\|^{\cdot2}_\alpha+2\langle\mu,\vartheta-\vartheta'\rangle^\cdot_\alpha\\
{}&=\int[\kappa_{\alpha/2}\mu]^2\,dm+2\int\kappa_{\alpha/2}\mu\cdot\kappa_{\alpha/2}(\vartheta-\vartheta')\,dm\in(-\infty,\infty).\notag\end{align}

The aim of this paper is to investigate the following problem.

\begin{problem}\label{pr3}Does there exist $\dot{\lambda}^\sigma_{\mathbf A}\in\ddot{\mathcal E}_\alpha^\sigma(\mathbf A,\mathbf1)$ with
\[\dot{G}_{\alpha,\vartheta-\vartheta'}(\dot{\lambda}^\sigma_{\mathbf A})=\dot{G}^\sigma_{\alpha,\vartheta-\vartheta'}(\mathbf A,\mathbf1):=\inf_{\mu\in\ddot{\mathcal E}_\alpha^\sigma(\mathbf A,\mathbf1)}\,\dot{G}_{\alpha,\vartheta-\vartheta'}(\mu)?\]\end{problem}

Note that the concept of weak Gauss integral extends to $\mu\in\dot{\mathcal E}_\alpha$ the concept of (standard) Gauss integral, defined for $\mu\in\mathcal E_\alpha$ by the formula \begin{align}\label{defG1}G_{\alpha,\vartheta-\vartheta'}(\mu)&:=\|\mu\|^2_\alpha+2\langle\mu,\vartheta-\vartheta'\rangle_\alpha\\
{}&=\int\kappa_\alpha\mu\,d\mu+2\int\kappa_\alpha(\vartheta-\vartheta')\,d\mu,\notag\end{align}
see \cite[Eq.~(4.5.14)]{L}. Indeed, since $\vartheta-\vartheta'\in\mathcal E_\alpha$, we obtain from (\ref{dot0})
\begin{equation}\label{rest}\dot{G}_{\alpha,\vartheta-\vartheta'}(\mu)=G_{\alpha,\vartheta-\vartheta'}(\mu)\text{ \ for all\ }\mu\in\mathcal E_\alpha.\end{equation}
In view of the terminology used in minimum energy problems with external fields \cite{O,ST,DS}, the interest to which was initially inspired by Gauss \cite{Gauss}, Problem~\ref{pr3} may therefore be called the {\it weak constrained Gauss variational problem}.

Also observe that
\[-\infty<\dot{G}^\sigma_{\alpha,\vartheta-\vartheta'}(\mathbf A,\mathbf1)<\infty,\]
the latter estimate being clear from $\ddot{\mathcal E}_\alpha^\sigma(\mathbf A,\mathbf1)\ne\varnothing$, and the former from
\[\dot{G}_{\alpha,\vartheta-\vartheta'}(\mu)=\|\mu+(\vartheta-\vartheta')\|^{\cdot2}_\alpha-\|\vartheta-\vartheta'\|^{\cdot2}_\alpha\geqslant
-\|\vartheta-\vartheta'\|^2_\alpha>-\infty\text{ \ for all\ }\mu\in\dot{\mathcal E}_\alpha.\]

\begin{lemma}\label{l-unique}A solution to Problem\/ {\rm\ref{pr3}} is unique\/ {\rm(}if it exists\/{\rm)}.\end{lemma}

\begin{proof}This follows by standard methods based on the convexity of $\ddot{\mathcal E}_\alpha^\sigma(\mathbf A,\mathbf1)$ and the parallelogram identity in the pre-Hil\-bert space $\dot{\mathcal E}_\alpha$. Indeed, if $\lambda$ and $\hat{\lambda}$ are two solutions to Problem~\ref{pr3}, then we get from (\ref{defG2})
\[4\dot{G}_{\alpha,\vartheta-\vartheta'}(\mathbf A,\mathbf1)\leqslant4\dot{G}_{\alpha,\vartheta-\vartheta'}\biggl(\frac{\lambda+\hat{\lambda}}{2}\biggr)=
\|\lambda+\hat{\lambda}\|_\alpha^{\cdot2}+4\langle\lambda+\hat{\lambda},\vartheta-\vartheta'\rangle^\cdot_\alpha.\]
On the other hand, applying the parallelogram identity in $\dot{\mathcal E}_\alpha$ to
$\lambda$ and $\hat{\lambda}$ and then adding and
subtracting $4\langle\lambda+\hat{\lambda},\vartheta-\vartheta'\rangle^\cdot_\alpha$ we obtain
\[\|\lambda-\hat{\lambda}\|_\alpha^{\cdot2}=-\|\lambda+\hat{\lambda}\|_\alpha^{\cdot2}-
4\langle\lambda+\hat{\lambda},\vartheta-\vartheta'\rangle^\cdot_\alpha+2\dot{G}_{\alpha,\vartheta-\vartheta'}(\lambda)+2\dot{G}_{\alpha,\vartheta-\vartheta'}(\hat{\lambda}).\]
When combined with the preceding relation, this yields
\[0\leqslant\|\lambda-\hat{\lambda}\|^{\cdot2}_\alpha\leqslant-4\dot{G}_{\alpha,\vartheta-\vartheta'}(\mathbf A,\mathbf1)+
2\dot{G}_{\alpha,\vartheta-\vartheta'}(\lambda)+2\dot{G}_{\alpha,\vartheta-\vartheta'}(\hat{\lambda})=0,\]
which establishes the identity $\lambda=\hat{\lambda}$ because $\|\cdot\|^\cdot_\alpha$ is a norm in $\dot{\mathcal E}_\alpha$.
\end{proof}

In this paper we obtain sufficient conditions for the existence of a solution $\dot{\lambda}_{\mathbf A}^\sigma$ to Problem~\ref{pr3} (Theorem~\ref{main1}) and prove their sharpness (Theorem~\ref{infcap}). We establish variational inequalities for the potential $\kappa_\alpha\dot{\lambda}^\sigma_{\mathbf A}$ (Theorem~\ref{main2}, Corollary~\ref{cor-main2}) and observe that for a condenser with separated plates, those inequalities determine the solution to Problem~\ref{pr3} uniquely among the admissible measures $\mu\in\ddot{\mathcal E}_\alpha^\sigma(\mathbf A,\mathbf1)$ (Theorem~\ref{main2-sep}). Treating the solution $\dot{\lambda}_{\mathbf A}^\sigma$ as a function of $(\mathbf A,\sigma)$, we analyze its continuity relative to the vague topology and the topologies determined by the weak and standard energy norms (Theorems~\ref{th-al-cont}, \ref{th-al-cont'}). We show that the above-men\-tion\-ed Theorem~\ref{main1} fails in general once Problem~\ref{pr3} is reformulated in the setting of standard energy, thereby justifying the need for the concept of weak energy when dealing with condensers with touching plates (see Section~\ref{sec-adv}).

\section{Relevant standard minimum energy problems}\label{sec-rel}

For any $\sigma\in\mathfrak C(A)\cup\{\infty\}$ and $\vartheta$ given by (\ref{g}), write
\begin{align*}\mathcal E^\sigma_g(A,1)&:=\mathcal E_g\cap\breve{\mathfrak M}^\sigma(A,1),\\
G_{g,\vartheta}^\sigma(A,1)&:=\inf_{\nu\in\mathcal E^\sigma_g(A,1)}\,G_{g,\vartheta}(\nu),\end{align*}
where
\begin{equation*}G_{g,\vartheta}(\nu):=\|\nu\|^2_g+2\langle\nu,\vartheta\rangle_g\in(-\infty,\infty).\end{equation*}
Here we have used the fact that $\vartheta\in\mathcal E_g$ by (\ref{g}).
It follows from (\ref{nonzero}) and (\ref{adm-c}) that the class $\mathcal E^\sigma_g(A,1)$ is nonempty, and hence the following problem makes sense.

\begin{problem}\label{pr1}Does there exist $\lambda^\sigma_A\in\mathcal E^\sigma_g(A,1)$ with
\[G_{g,\vartheta}(\lambda^\sigma_A)=G^\sigma_{g,\vartheta}(A,1)?\]\end{problem}

A key observation behind the analysis of Problem~\ref{pr3}, performed in this study, is that Problem~\ref{pr3} can be reduced to Problem~\ref{pr1} whenever (\ref{suff}) holds (see Theorem~\ref{main1} and Lemma~\ref{l-ineq}); while this auxiliary problem can already be treated relatively easily because of the positivity of the admissible measures $\nu\in\mathcal E^\sigma_g(A,1)$ and the perfectness of the $\alpha$-Green kernel $g$ (see Section~\ref{sec:aux} for the results obtained).

Now we shall show that for a condenser with separated plates, Problem~\ref{pr3} is equivalent to the (standard) constrained Gauss variational problem of minimizing  $G_{\alpha,\vartheta-\vartheta'}(\mu)$ over the (nonempty) class $\mathcal E_\alpha^\sigma(\mathbf A,\mathbf1)$; for the notation, see (\ref{stG}) and~(\ref{defG1}).

\begin{problem}\label{pr4}Does there exist $\lambda^\sigma_{\mathbf A}\in\mathcal E_\alpha^\sigma(\mathbf A,\mathbf1)$ with\footnote{Note that a solution to Problem~\ref{pr4}, resp.\ Problem~\ref{pr1}, is unique (if it exists). This can be seen similarly as in proof of Lemma~\ref{l-unique} by applying convexity arguments and the parallelogram identity in the pre-Hil\-bert space $\mathcal E_\alpha$, resp.\ $\mathcal E_g$.}
\begin{equation}\label{e-pr4}G_{\alpha,\vartheta-\vartheta'}(\lambda^\sigma_{\mathbf A})=G^\sigma_{\alpha,\vartheta-\vartheta'}(\mathbf A,\mathbf1):=\inf_{\mu\in\mathcal E_\alpha^\sigma(\mathbf A,\mathbf1)}\,G_{\alpha,\vartheta-\vartheta'}(\mu)?\end{equation}\end{problem}

\begin{lemma}\label{l-sep}If the separation condition\/ {\rm(\ref{dist})} holds, then
\begin{align}\label{ext1}\mathcal E_\alpha^\sigma(\mathbf A,\mathbf1)&=\ddot{\mathcal E}_\alpha^\sigma(\mathbf A,\mathbf1),\\
\label{ext2}G^\sigma_{\alpha,\vartheta-\vartheta'}(\mathbf A,\mathbf1)&=\dot{G}^\sigma_{\alpha,\vartheta-\vartheta'}(\mathbf A,\mathbf1).
\end{align}
\end{lemma}

\begin{proof}Fix $\mu\in\mathfrak M(\mathbf A,\mathbf 1)$. By the Riesz composition identity and Fubini's
theorem,
\begin{align*}\int\kappa_{\alpha/2}\mu^+\kappa_{\alpha/2}\mu^-\,dm
 &=\int\left(\int\kappa_{\alpha/2}(x,y)\,d\mu^+(y)\right)
 \left(\int\kappa_{\alpha/2}(x,z)\,d\mu^-(z)\right)\,dm(x)\\
 &=\int\left(\int|x-y|^{\alpha/2-n}|x-z|^{\alpha/2-n}\,dm(x)\right)\,d(\mu^+\otimes\mu^-)(y,z)\\
 &=\int\left(\int|x-y+z|^{\alpha/2-n}|x|^{\alpha/2-n}\,dm(x)\right)\,d(\mu^+\otimes\mu^-)(y,z)\\
 &=\int\kappa_\alpha(y,z)\,d(\mu^+\otimes\mu^-)(y,z)
 \leqslant\bigl[\text{\rm dist}(A,F)\bigr]^{\alpha-n}<\infty.
\end{align*}
Assuming now $\mu\in\dot{\mathcal E}_\alpha^\sigma(\mathbf A,\mathbf 1)$, we obtain
\[\bigl(\kappa_{\alpha/2}\mu^+\bigr)^2+\bigl(\kappa_{\alpha/2}\mu^-\bigr)^2=
\bigl(\kappa_{\alpha/2}\mu\bigr)^2
+2\kappa_{\alpha/2}\mu^+\,\kappa_{\alpha/2}\mu^-\in L^1(m),\]
hence $\kappa_{\alpha/2}\mu^+,\kappa_{\alpha/2}\mu^-\in L^2(m)$. In view of (\ref{w-incl}), this implies $\mu^+,\mu^-\in\dot{\mathcal E}^+_\alpha=\mathcal E^+_\alpha$, and so $\mu\in\mathcal E_\alpha^\sigma(\mathbf A,\mathbf 1)$. As $\mathcal E_\alpha^\sigma(\mathbf A,\mathbf 1)\subset\dot{\mathcal E}_\alpha^\sigma(\mathbf A,\mathbf 1)$ is obvious,
\[\mathcal E^\sigma_\alpha(\mathbf A,\mathbf 1)=\dot{\mathcal E}^\sigma_\alpha(\mathbf A,\mathbf 1).\]
On the other hand, it follows easily from (\ref{def_E2}) that
\[\mathcal E_\alpha^\sigma(\mathbf A,\mathbf 1)\subset\ddot{\mathcal E}^\sigma_\alpha(\mathbf A,\mathbf 1)\subset\dot{\mathcal E}^\sigma_\alpha(\mathbf A,\mathbf 1),\]
which combined with the preceding display establishes (\ref{ext1}). Since the weak and standard Gauss integrals for $\mu\in\mathcal E_\alpha$ coincide (see (\ref{rest})), (\ref{ext1}) implies (\ref{ext2}).
\end{proof}

\begin{corollary}\label{cor-sep} If the separation condition\/ {\rm(\ref{dist})} holds, then Problems\/~{\rm\ref{pr3}} and\/ {\rm\ref{pr4}} are simultaneously either solvable or unsolvable, and in the affirmative case their solutions coincide: $\dot{\lambda}^\sigma_{\mathbf A}=\lambda^\sigma_{\mathbf A}$.\end{corollary}

\begin{proof}This is obvious in view of (\ref{rest}), (\ref{ext1}), and (\ref{ext2}).\end{proof}

\begin{remark}However, for condensers with touching plates ($\text{\rm dist}(A,F)=0$), which we are mainly interested in, Corollary~\ref{cor-sep} on the simultaneous solvability or unsolvability of Problems~\ref{pr3} and \ref{pr4} fails in general; see Section~\ref{sec-adv} below for details.\end{remark}

\begin{figure}[htbp]
\begin{center}
\vspace{-.4in}
\hspace{1in}\includegraphics[width=4in]{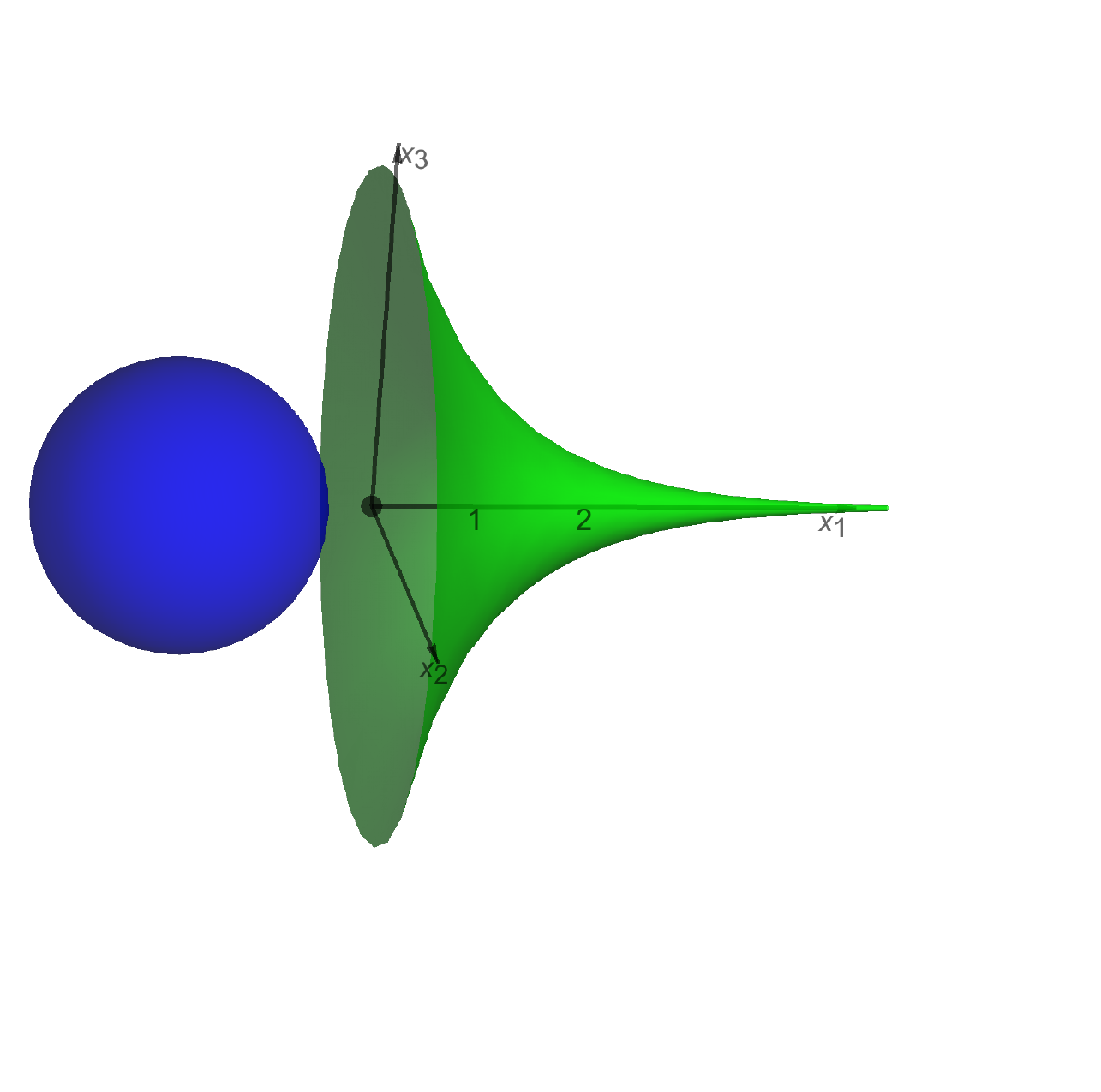}
\vspace{-.8in}
\caption{$\mathbf A:=(A,F)$ in $\mathbb R^3$, where $F=\bigl\{0\leqslant x_1<\infty,\ x_2^2+x_3^2\leqslant\varrho^2(x_1)\bigr\}$
with $\varrho(x_1)=\exp(-x_1)$ and $A$ is a closed ball in $F^c$.\vspace{-.1in}}
\label{Fig2}
\end{center}
\end{figure}

\begin{example}\label{rem-ex}Let $n=3$, $\alpha=2$, $\sigma=\infty$, $\vartheta=0$, $F:=D^c:=Q_\varrho$, and $\mathbf A:=(A,F)$, where $Q_\varrho$ is given by (\ref{descr}) and $A$ is a closed ball in $F^c$. Then Problem~\ref{pr3} is equivalent to Problem~\ref{pr4} by Corollary~\ref{cor-sep}, while the latter reduces to Problem~\ref{pr2}. Combining \cite[Theorem~5]{Z1}, providing necessary and sufficient conditions for the solvability of Problem~\ref{pr2}, with Example~\ref{ex} above implies that the solution to Problem~\ref{pr3} (equivalently, Problem~\ref{pr2}) with the stated data does exist if $\varrho$ is given by either (\ref{c1}) or (\ref{c3}), while these two problems have {\it no\/} solution
if $\varrho$ is defined by (\ref{c2}) (see Figure~\ref{Fig2}). These theoretical results have been illustrated in \cite{HWZ,OWZ}
by means of numerical experiments.\end{example}

\section{Main results}\label{sec-main}

As always, the permanent assumptions stated in Sections~\ref{sec:intr}--\ref{sec:cond} are required to hold. Throughout the present section we also assume that $F$ (${}=D^c$) is not $\alpha$-thin at infinity. The proofs of the assertions formulated below will be given in Section~\ref{sec-proofs}.

\subsection{On the solvability of Problem~\ref{pr3}} We begin with sufficient and/or necessary conditions for the existence of solutions to Problem~\ref{pr3}.

\begin{theorem}\label{main1} Assume that\/\footnote{If the separation condition (\ref{dist}) holds, then $c_g(A)<\infty$ and $c_\alpha(A)<\infty$ are actually equivalent. This is seen by applying (\ref{10}) to any $\nu\in\breve{\mathfrak M}^+(A,1)$.}
\begin{equation}\label{suff}c_g(A)<\infty\text{ \ or \ }\sigma(A)<\infty.\end{equation}
Then there is the\/ {\rm(}unique\/{\rm)} solution\/ $\dot{\lambda}^\sigma_{\mathbf A}$ to Problem\/~{\rm\ref{pr3}}, and moreover
\begin{equation}\label{th1-1}\dot{\lambda}_{\mathbf A}^\sigma=\lambda_A^\sigma-(\lambda_A^\sigma)',\end{equation}
where\/ $\lambda_A^\sigma$ is the solution to Problem\/~{\rm\ref{pr1}} and\/ $(\lambda_A^\sigma)'$ its\/ $\alpha$-Riesz balayage onto\/ $F$.
\end{theorem}

\begin{example}\label{ex1}Let $A:=D:=\{x\in\mathbb R^n:\ |x|<1\}$ and $\sigma:=m|_D$. Then for any $\alpha\in(0,2]$, $D^c$ is not $\alpha$-thin at infinity and (\ref{adm-c}) holds. Since $\sigma(A)<\infty$, it follows from Theorem~\ref{main1} that Problem~\ref{pr3} with these data is solvable for any $\vartheta$ given by (\ref{g}). Observe that the solvability occurs despite the fact that the plates $A$ and $D^c$ touch each other over the whole sphere $\{x:\ |x|=1\}$.\end{example}

The following assertion shows that the sufficient conditions for the solvability of Problem~\ref{pr3},
established in Theorem~\ref{main1}, are actually {\it sharp}.

\begin{theorem}\label{infcap} Assume that\/ $\vartheta=0$ and\/ $c_g(A)=\infty$.
Then Problem\/~{\rm\ref{pr3}} is unsolvable for every\/ $\sigma\in\mathfrak C(A)\cup\{\infty\}$ such that\/ $\sigma\geqslant\sigma_0$,
where\/ $\sigma_0\in\mathfrak C(A)$ with\/ $\sigma_0(A)=\infty$ is chosen suitably.
\end{theorem}

Combining Theorems~\ref{main1} and \ref{infcap} leads to the following corollary.

\begin{corollary}\label{lusin}Assume that\/ $\vartheta=0$. Then the following assertions are equivalent.
\begin{itemize}\item[{\rm (i)}] $c_g(A)<\infty$.
\item[{\rm (ii)}] Problem\/~{\rm\ref{pr3}} is solvable for\/ $\sigma=\infty$.
\item[{\rm (iii)}] Problem\/~{\rm\ref{pr3}} is solvable for every\/ $\sigma\in\mathfrak C(A)\cup\{\infty\}$.
\end{itemize}
In any of these\/ {\rm(i)--(iii)} holds, then necessarily
\begin{equation}\label{e-l}c_\alpha(\partial D\cap\text{\rm Cl}_{\mathbb R^n}A)=0.\end{equation}
\end{corollary}

\begin{remark} Observe that the condensers $\mathbf A$ such that Problem~\ref{pr3} is solvable in the constrained setting ($\sigma\in\mathfrak C(A)$)
differ drastically from those for which the solvability occurs in the unconstrained setting ($\sigma=\infty$). Indeed, if $\sigma\in\mathfrak C(A)$ is bounded, then according to Theorem~\ref{main1}, Problem~\ref{pr3} is solvable even if
\[c_\alpha(\partial D\cap\text{\rm Cl}_{\mathbb R^n}A)>0\]
(actually, even if $A=D$, see Example~\ref{ex1}); compare with \cite{Bec,DFHSZ1}. However, if $\sigma=\infty$ and $\vartheta=0$ (no constraint and no external source of energy), then the solvability of Problem~\ref{pr3} necessarily implies (\ref{e-l}) (see Corollary~\ref{lusin}).\end{remark}

\subsection{Variational inequalities for the potential $\kappa_\alpha\dot{\lambda}^\sigma_{\mathbf A}$}\label{sec-decr} In Section~\ref{sec-decr} we fix $\sigma\in\mathfrak C(A)$ and assume that $\kappa_\alpha\vartheta^-|_A$ is upper bounded.  Write
\begin{equation}\label{def-f}f:=\kappa_\alpha(\vartheta-\vartheta').\end{equation}

\begin{theorem}\label{main2} Assume that\/ {\rm(\ref{suff})} holds. Then the solution\/ $\dot{\lambda}:=\dot{\lambda}^\sigma_{\mathbf A}$ to Problem\/~{\rm\ref{pr3}} {\rm(}which exists according to Theorem\/~{\rm\ref{main1}}{\rm)} satisfies the relations\/\footnote{If $\vartheta\geqslant0$, then actually $w>0$. Furthermore, then $\kappa_\alpha\dot{\lambda}+f$ is l.s.c.\ on $D$, being equal to $g(\dot{\lambda}^++\vartheta)$ on $D$, and (\ref{desc2a}) takes the following apparently stronger form: $\kappa_\alpha\dot{\lambda}+f\leqslant w$ on $S^{\dot{\lambda}^+}_D$.}
\begin{align}\label{desc0a}&\dot{\lambda}^-=(\dot{\lambda}^+)',\\
\label{desc1a}&\kappa_\alpha\dot{\lambda}+f\geqslant w\quad(\sigma-\dot{\lambda}^+)\mbox{-a.e.},\\
&\kappa_\alpha\dot{\lambda}+f\leqslant w\quad\dot{\lambda}^+\mbox{-a.e.}\label{desc2a}
\end{align}
with some\/ $w\in\mathbb R$.\end{theorem}

\begin{corollary}\label{cor-main2}Assume that both\/ {\rm(\ref{suff})} and\/ $\vartheta=0$ hold. Then the solution\/ $\dot{\lambda}:=\dot{\lambda}^\sigma_{\mathbf A}$ to Problem\/~{\rm\ref{pr3}} satisfies\/ {\rm(\ref{desc0a})} as well as
\begin{align}\label{desc1''}\kappa_\alpha\dot{\lambda}&=w\quad(\sigma-\dot{\lambda}^+)\mbox{-a.e.},\\
\kappa_\alpha\dot{\lambda}&\leqslant w\quad{on\ }\mathbb R^n\label{desc2''}\end{align}
with some\/ $w\in(0,\infty)$. If, moreover, $\alpha<2$ and\/ $m(D^c)>0$, then also
\begin{equation}\label{s1}
S_D^{\dot{\lambda}^+}=S_D^{\sigma}.
\end{equation}
\end{corollary}

For a condenser with separated plates, relations (\ref{desc0a})--(\ref{desc2a}) determine the solution to Problem~\ref{pr3} uniquely among the admissible measures $\mu\in\ddot{\mathcal E}_\alpha^\sigma(\mathbf A,\mathbf1)$. To be exact, the following assertion is valid.

\begin{theorem}\label{main2-sep} Assume that the separation condition\/ {\rm(\ref{dist})} holds. Then any given measure\/ $\dot{\lambda}\in\ddot{\mathcal E}_\alpha^\sigma(\mathbf A,\mathbf1)$ serves as a solution to Problem\/~{\rm\ref{pr3}} if and only if it satisfies relations\/ {\rm(\ref{desc0a})}--{\rm(\ref{desc2a})} with some\/ $w\in\mathbb R$.
\end{theorem}

\begin{remark}Under the hypotheses of Corollary~\ref{cor-main2}, assume moreover $\sigma(A)<\infty$. Then the number $w$ appearing in (\ref{desc1''}) and (\ref{desc2''}) can be written in the form
\[w=\frac{E_g(\dot{\lambda}^+,\sigma-\dot{\lambda}^+)}{\sigma(A)-1}\in(0,\infty).\]
Thus $E_g(\dot{\lambda}^+,\sigma-\dot{\lambda}^+)<\infty$, though $E_g(\sigma)$ may be infinite.
\end{remark}

\subsection{On continuity of $\dot{\lambda}_{\mathbf A}^\sigma$ when $\mathbf A$ and $\sigma$ vary}\label{sec-cont} Treating the solution $\dot{\lambda}_{\mathbf A}^\sigma$ to Problem~\ref{pr3} as a function of $(\mathbf A,\sigma)$, in Theorem~\ref{th-al-cont} below we establish its continuity relative to the topology determined by the weak energy norm on $\dot{\mathcal E}_\alpha$.

Given a condenser $\mathbf A=(A,F)$ and a constraint $\sigma\in\mathfrak C(A)\cup\{\infty\}$, consider a decreasing sequence (net) $(A_k)$ of relatively closed subsets of $D$ whose intersection equals $A$, and a sequence (net) $(\sigma_k)$ of constraints $\sigma_k\in\mathfrak C(A_k)\cup\{\infty\}$ such that
\[\sigma_p\geqslant\sigma_k\geqslant\sigma\text{ \ for all\ }k\geqslant p,\]
and in the case $\sigma\ne\infty$ we have $\sigma_k\ne\infty$ for all $k$ and moreover
\[\sigma_k\to\sigma\text{ \ vaguely in\ }\mathfrak M(D).\]

\begin{theorem}\label{th-al-cont} Under these requirements, assume in addition that
\begin{equation}c_g(A_1)<\infty\text{ \ or \ }\sigma_1(A_1)<\infty.\label{c1'}\end{equation}
Then
\begin{align}\label{g-cont1'}&\lim_{k}\,\dot{G}_{\alpha,\vartheta-\vartheta'}^{\sigma_k}(\mathbf A_k,\mathbf1)=\dot{G}_{\alpha,\vartheta-\vartheta'}^\sigma(\mathbf A,\mathbf1),\\
\label{g-cont2'}&\lim_{k}\,\|\dot{\lambda}^{\sigma_k}_{\mathbf A_k}-\dot{\lambda}^\sigma_{\mathbf A}\|_\alpha^\cdot=0,
\end{align}
where\/ $\mathbf A_k:=(A_k,F)$ and\/ $\dot{\lambda}^{\sigma_k}_{\mathbf A_k}$, resp.\ $\dot{\lambda}^\sigma_{\mathbf A}$, is the solution to Problem\/~{\rm\ref{pr3}} with\/ $\mathbf A_k$ and\/ $\sigma_k$, resp.\ $\mathbf A$ and\/ $\sigma$. {\rm(}Those\/ $\dot{\lambda}^{\sigma_k}_{\mathbf A_k}$ and\/ $\dot{\lambda}^\sigma_{\mathbf A}$ exist in consequence of Theorem\/~{\rm\ref{main1}}.{\rm)}
\end{theorem}

Requiring now additionally that the condensers in question have separated plates, we further establish the continuity of the solution to Problem~\ref{pr3} relative to the vague topology on $\mathfrak M(\mathbb R^n)$ as well as the strong topology on $\mathcal E_\alpha$, determined by the standard energy norm.

\begin{theorem}\label{th-al-cont'}Under the hypotheses of Theorem\/~{\rm\ref{th-al-cont}}, if moreover\/
\begin{equation}\label{sep}\text{\rm dist}(A_1,F)>0,\end{equation}
then\/ {\rm(\ref{g-cont1'})} and\/ {\rm(\ref{g-cont2'})} take the equivalent form
\begin{align}\label{g-cont1''}&\lim_{k}\,G_{\alpha,\vartheta-\vartheta'}^{\sigma_k}(\mathbf A_k,\mathbf1)=G_{\alpha,\vartheta-\vartheta'}^\sigma(\mathbf A,\mathbf1),\\
\label{g-cont2''}&\lim_{k}\,\|\lambda^{\sigma_k}_{\mathbf A_k}-\lambda^\sigma_{\mathbf A}\|_\alpha=0,
\end{align}
where\/ $\lambda^{\sigma_k}_{\mathbf A_k}$ and\/ $\lambda^\sigma_{\mathbf A}$ are the solutions to Problem\/~{\rm\ref{pr4}} {\rm(}equivalently, to Problem\/~{\rm\ref{pr3}}; see Corollary\/~{\rm\ref{cor-sep}}{\rm)} with the stated data. Furthermore, then
\begin{equation}\label{g-cont3'}\bigl(\lambda^{\sigma_k}_{\mathbf A_k}\bigr)^\pm\to\bigl(\lambda^\sigma_{\mathbf A}\bigr)^\pm\text{ \ vaguely in\ }\mathfrak M(\mathbb R^n).\end{equation}\end{theorem}

\section{About (auxiliary) Problem~\ref{pr1}}\label{sec:aux}

\subsection{Criteria for the solvability of Problem~\ref{pr1}} The main tool of our analysis of Problem~\ref{pr1} is a strong completeness result for a topological subspace of $\mathcal E_g$, established in Lemma~\ref{lemma-aux} below. In turn, its proof is based on the perfectness of the $\alpha$-Green kernel $g$, discovered recently in \cite{FZ} (see Section~\ref{sec-gr} above).

\begin{lemma}\label{lemma-aux} Assume\/ {\rm(\ref{suff})} holds. Then the cone\/ $\mathcal E^\sigma_g(A,1)$ is complete in the strong topology on\/ $\mathcal E_g$, determined by\/ $\|\cdot\|_g$. In more detail, any strong Cauchy sequence\/ {\rm(}net\/{\rm)} $(\nu_j)\subset\mathcal E^\sigma_g(A,1)$ converges both
strongly and vaguely to a unique\/ $\nu_0\in\mathcal
E^\sigma_g(A,1)$.
\end{lemma}

\begin{proof}Fix a strong Cauchy sequence $(\nu_j)\subset\mathcal E^\sigma_g(A,1)$; then it is strongly bounded:
\begin{equation}\label{bound}\sup_{j\in\mathbb N}\,\|\nu_j\|_g<\infty.\end{equation}
Being obviously vaguely bounded, it has a vague cluster point $\nu_0\in\mathfrak M^+(D)$ \cite[Chapter~III, Section~2, Proposition~9]{B2}.
Moreover, $\nu_0\in\mathcal E^+_g$ because the energy $E_g(\cdot)$ is vaguely l.s.c.\ on $\mathfrak M^+(D)$ \cite[Lemma~2.2.1(e)]{F1}. Since the kernel $g$ is perfect, $(\nu_j)$ converges both strongly and vaguely to $\nu_0$, and this $\nu_0$ is unique. Furthermore, $\nu_0$ is carried by $A$, because the cone of all $\nu\in\mathfrak M^+(D)$ carried by a relatively closed subset of $D$ is vaguely closed in $\mathfrak M(D)$. Finally, $\nu_0\leqslant\sigma$, the vague limit of positive measures likewise being positive.

It remains therefore to show that under condition (\ref{suff}),
\begin{equation}\label{b}\nu_0(A)=1.\end{equation}
Choose an increasing sequence $(K_k)$ of compact sets with the union $A$. Since the indicator function $1_{K_k}$ of $K_k$ is upper
semicontinuous on $D$, while $1_D$ is (finitely) continuous
on $D$, we obtain from Lemma~\ref{lemma-semi} applied
subsequently to $1_D$ and $-1_{K_k}$
\begin{align*}1=\lim_{j\to\infty}\,\nu_j(A)\geqslant\nu_0(A)&=\lim_{k\to\infty}\,\nu_0(K_k)\geqslant
\lim_{k\to\infty}\,\limsup_{j\to\infty}\,\nu_j(K_k)\\&{}=
1-\lim_{k\to\infty}\,\liminf_{j\to\infty}\,\nu_j(A\setminus
K_k).\end{align*}
Equality (\ref{b}) will therefore follow if we prove the relation
\begin{equation}\label{l}\lim_{k\to\infty}\,\liminf_{j\to\infty}\,\nu_j(A\setminus
K_k)=0.\end{equation}

Assume first that the constraint $\sigma$ is bounded. Since
\[\infty>\sigma(A)=\lim_{k\to\infty}\,\sigma(K_k),\]
we have
\[\lim_{k\to\infty}\,\sigma(A\setminus K_k)=0.\]
Combining this with
\[\nu_j(A\setminus K_k)\leqslant\sigma(A\setminus K_k)\text{ \ for all\ }j,k\in\mathbb N\]
gives (\ref{l}).

Assuming now $c_g(A)<\infty$, write $\check{K}_k:=A\setminus K_k$. According to Theorem~\ref{th-equi}, there exists the $g$-equ\-il\-ib\-rium measure $\gamma_k:=\gamma_{\check{K}_k}$ for $\check{K}_k$, and this $\gamma_k$ solves the problem of minimizing $E_g(\nu)$ over the convex cone $\Gamma_k:=\Gamma_{\check{K}_k}$ of all $\nu\in\mathcal E_g$ with $g\nu\geqslant1$ n.e.\ on $\check{K}_k$. In view of the monotonicity
of $(\check{K}_k)$, we observe from (\ref{eq2}) with $Q=\check{K}_k$ that $\gamma_k\in\Gamma_p$ for all $p\geqslant k$. Therefore, by \cite[Lemma~4.1.1]{F1},
\[
\|\gamma_k-\gamma_p\|^2_g\leqslant\|\gamma_k\|^2_g-\|\gamma_p\|^2_g\text{ \ for all\ }p\geqslant k.\]
Furthermore, it is clear from (\ref{eq1})  with $Q=\check{K}_k$ that the sequence
$(\|\gamma_k\|_g)$ is
bounded and nonincreasing, and hence it is Cauchy in $\mathbb R$. This together with the
preceding inequality shows that $(\gamma_k)$ is strong Cauchy in $\mathcal
E^+_g$. Since it obviously converges vaguely in $\mathfrak M(D)$ to zero, another application of the perfectness of $g$ implies that $\gamma_k\to0$ also strongly in $\mathcal
E^+_g$. Hence,
\begin{equation*}
\lim_{k\to\infty}\,\|\gamma_k\|_g=0.
\label{27}
\end{equation*}
Besides, by the Cauchy--Schwarz (Bunyakovski) inequality in $\mathcal E_g$,
\[\nu_j(\check{K}_k)=\int1_{\check{K}_k}\,d\nu_j\leqslant\int g\gamma_k\,d\nu_j=\langle
\gamma_k,\nu_j\rangle_g\leqslant\|\gamma_k\|_g\cdot\|\nu_j\|_g\text{ \ for all\
}j,k,\]
where the former inequality is obtained from (\ref{eq2}) with $Q=\check{K}_k$ in view of the fact that, being of zero capacity, the Borel set $\{x\in\check{K}_k:\ g\gamma_k(x)<1\}$ cannot carry $\nu_j\in\mathcal E_g^+$. Combining the last two displays with (\ref{bound}) gives (\ref{l}).\end{proof}

\begin{theorem}\label{th-g}If\/ {\rm(\ref{suff})} holds, there is the\/ {\rm(}unique\/{\rm)} solution\/ $\lambda^\sigma_{A}$ to Problem\/~{\rm\ref{pr1}}.\end{theorem}

\begin{proof}Since obviously
\begin{equation}\label{lb}G_{g,\vartheta}(\nu)=\|\nu+\vartheta\|_g^2-\|\vartheta\|_g^2\geqslant-\|\vartheta\|_g^2>-\infty\text{ \ for all\ }\nu\in\mathcal E_g,\end{equation}
$G^\sigma_{g,\vartheta}(A,1)$ is finite. Choose a sequence $(\nu_j)\subset\mathcal E^\sigma_g(A,1)$ so that
\begin{equation*}\lim_{j\to\infty}\,G_{g,\vartheta}(\nu_j)=G^\sigma_{g,\vartheta}(A,1).\end{equation*}
Based on the convexity of $\mathcal E^\sigma_g(A,1)$, we obtain by the parallelogram identity in $\mathcal E_g$
\begin{equation*}0\leqslant\|\nu_j-\nu_k\|^2_g\leqslant-4G^\sigma_{g,\vartheta}(A,1)+
2G_{g,\vartheta}(\nu_j)+2G_{g,\vartheta}(\nu_k)\text{ \ for all\ }j,k\in\mathbb N,\end{equation*}
which combined with the preceding display implies that $(\nu_j)$ is strong Cauchy in $\mathcal
E^\sigma_g(A,1)$. Therefore, according to Lemma~\ref{lemma-aux},
$(\nu_j)$ converges (both vaguely and) strongly to some
$\nu_0\in\mathcal E^\sigma_g(A,1)$. Having observed from the equality in (\ref{lb}) that the mapping $\nu\mapsto G_{g,\vartheta}(\nu)$ is strongly continuous on $\mathcal E_g$, we thus get
\[G^\sigma_{g,\vartheta}(A,1)\leqslant G_{g,\vartheta}(\nu_0)=\lim_{j\to\infty}\,G_{g,\vartheta}(\nu_j)=
G^\sigma_{g,\vartheta}(A,1),\]
and hence $\nu_0=:\lambda^\sigma_A$ is the required solution to Problem~\ref{pr1}.\end{proof}

\begin{remark}\label{rem-g}If $\sigma=\infty$ and $\vartheta=0$ (no constraint and no external source of energy), then Problem~\ref{pr1} reduces to the problem
\[\inf_{\nu\in\mathcal E_g\cap\breve{\mathfrak M}^+(A,1)}\,\|\nu\|^2_g\quad\bigl({}=1/c_g(A)\bigr).\]
If (and only if) $c_g(A)<\infty$, then this problem is solvable, and its (unique) solution equals $\gamma_A/c_g(A)$, $\gamma_A$ being the $g$-equ\-il\-ibr\-ium measure on $A$, cf.\ Theorem~\ref{th-equi}.
\end{remark}

\subsection{On continuity of $\lambda^\sigma_A$ when $A$ and $\sigma$ vary} Treating the solution $\lambda^\sigma_A$ to Problem~\ref{pr1} as a function of $(A,\sigma)$, we shall now analyze its continuity relative to both the strong and vague topologies on $\mathcal E^+_g$.

\begin{lemma}\label{lequiv} A measure\/ $\lambda\in\mathcal E_g^\sigma(A,1)$
solves Problem\/~{\rm\ref{pr1}} if and only if
\begin{equation}\label{lchar}\langle\lambda+\vartheta,\nu-\lambda\rangle_g\geqslant0\mbox{ \ for all\ }
\nu\in\mathcal E_g^\sigma(A,1).\end{equation}
\end{lemma}

\begin{proof}By direct calculation, for any $\lambda,\nu\in\mathcal E_g^\sigma(A,1)$ and
$h\in(0,1]$ we obtain
\begin{equation}\label{mainin}G_{g,\vartheta}\bigl(h\nu+(1-h)\lambda\bigr)-G_{g,\vartheta}(\lambda)=2h\langle
\lambda+\vartheta,\nu-\lambda\rangle_g+h^2\|\nu-\lambda\|^2_g.\end{equation}
If $\lambda$ solves Problem~\ref{pr1}, then in view of the convexity of the class $\mathcal E_g^\sigma(A,1)$, the left-hand (hence, the
right-hand) side of (\ref{mainin}) is ${}\geqslant0$, which leads to (\ref{lchar}) by
letting $h\to0$. Conversely, if (\ref{lchar}) holds, then
(\ref{mainin}) with $h=1$ gives $G_{g,\vartheta}(\nu)\geqslant
G_{g,\vartheta}(\lambda)$ for all $\nu\in\mathcal E_g^\sigma(A,1)$, and hence $\lambda$ solves, indeed, Problem~\ref{pr1}.\end{proof}

Let $(A_k)$ and $(\sigma_k)$ be as described in the beginning of Section~\ref{sec-cont}. Since $\mathcal E_g^\sigma(A,1)$ is nonempty by our permanent requirements, so are all the $\mathcal E_g^{\sigma_k}(A_k,1)$, cf.\ (\ref{bel}). Assuming additionally that (\ref{c1'}) holds, we then imply by Theorem~\ref{th-g} that there exists the solution $\lambda^\sigma_A$, resp.\ $\lambda^{\sigma_k}_{A_k}$, to Problem~\ref{pr1} with $\sigma$ and $A$, resp.\ $\sigma_k$ and $A_k$.

\begin{theorem}\label{th-g-cont} Under these assumptions, we have
\begin{align}\label{g-cont1}&\lim_{k}\,G_{g,\vartheta}^{\sigma_k}(A_k,1)=G_{g,\vartheta}^\sigma(A,1),\\
\label{g-cont2}&\lim_{k}\,\|\lambda^{\sigma_k}_{A_k}-\lambda^\sigma_A\|_g=0,\\
\label{g-cont3}&\lambda^{\sigma_k}_{A_k}\to\lambda^\sigma_A\text{ \ vaguely in\ }\mathfrak M(D).\end{align}
\end{theorem}

\begin{proof}Write $\lambda_k:=\lambda^{\sigma_k}_{A_k}$. It is seen from the monotonicity of $(A_k)$ and $(\sigma_k)$ that
\begin{equation}\label{bel}\mathcal E_g^\sigma(A,1)\subset\mathcal E_g^{\sigma_k}(A_k,1)\subset\mathcal E_g^{\sigma_p}(A_p,1)\text{ \ for all\ }k\geqslant p,\end{equation}
which in view of (\ref{lb}) implies
\begin{equation}\label{f'}-\infty<\lim_{k}\,G_{g,\vartheta}(\lambda_k)\leqslant G_{g,\vartheta}^\sigma(A,1)<\infty.\end{equation}
As $\lambda_k\in\mathcal E_g^{\sigma_p}(A_p,1)$ for all $k\geqslant p$, we get by applying (\ref{lchar}) to $\lambda=\lambda_p$ and $\nu=\lambda_k$
\[\langle\lambda_p+\vartheta,\lambda_k-\lambda_p\rangle_g\geqslant0,\]
which combined with (\ref{mainin}) taken for $h=1$, $\lambda=\lambda_p$, and $\nu=\lambda_k$ gives
\begin{equation*}\label{fund}\|\lambda_k-\lambda_p\|^2_g\leqslant G_{g,\vartheta}(\lambda_k)-G_{g,\vartheta}(\lambda_p).\end{equation*}
This together with (\ref{f'}) yields that for every $p$, $(\lambda_k)_{k\geqslant p}$ is strong Cauchy in $\mathcal E_g^{\sigma_p}(A_p,1)$. By Lemma~\ref{lemma-aux}, there is therefore the unique $\nu_0$ such that
\begin{equation}\label{f}\lambda_k\to\nu_0\text{ \ in the strong and vague topologies on\ }\mathcal E^+_g\end{equation}
and which belongs to $\mathcal E_g^{\sigma_p}(A_p,1)$ for all $p$, and hence to their intersection:
\[\nu_0\in\bigcap_{p}\,\mathcal E_g^{\sigma_p}(A_p,1)\subset\mathcal E_g\cap\breve{\mathfrak M}^+(A,1).\]
Since $\sigma_k-\lambda_k\geqslant0$ and $\sigma_k-\lambda_k\to\sigma-\nu_0$ vaguely whenever $\sigma\ne\infty$, we have $\nu_0\leqslant\sigma$, which together with the last display shows that, actually, $\nu_0\in\mathcal E_g^\sigma(A,1)$. As the map $\nu\mapsto G_{g,\vartheta}(\nu)$ is strongly continuous on $\mathcal E_g$, we obtain from (\ref{f'}) and (\ref{f})
\[G^\sigma_{g,\vartheta}(A,1)\leqslant G_{g,\vartheta}(\nu_0)=\lim_{k}\,G_{g,\vartheta}(\lambda_k)\leqslant G_{g,\vartheta}^\sigma(A,1),\]
which results in (\ref{g-cont1}) and the equality $\nu_0=\lambda^\sigma_A$. Finally, substituting $\nu_0=\lambda^\sigma_A$ into (\ref{f}) gives (\ref{g-cont2}) and (\ref{g-cont3}).
\end{proof}

\begin{remark}Theorems~\ref{th-g} and \ref{th-g-cont} remain valid (with pretty much the same proofs) if Problem~\ref{pr1} is replaced by a more general minimum energy problem
\[\inf_{\nu\in\mathcal E_g^\sigma(A,1)}\,\Bigl[\|\nu\|^2_g+2\int u\,d\nu\Bigr],\]
where either $u=g\chi$ with $\chi\in\mathcal E_g$, or $u:D\to[0,\infty]$ is l.s.c.\end{remark}

\subsection{Variational inequalities for the potential $g\lambda^\sigma_A$} The aim of this section is to establish necessary and sufficient conditions for the solvability of Problem~\ref{pr1} in terms of variational inequalities for the potential $g\lambda^\sigma_A$.\footnote{Assertions similar to Theorem~\ref{desc-th2} can be found e.g.\ in \cite{DFHSZ2,FZ-Pot2}. The first results of this kind have been established in \cite{R,DS} for the logarithmic kernel on the plane.}

\begin{theorem}\label{desc-th2}Fix\/ $\sigma\in\mathfrak C(A)$, and assume that\/ $g\vartheta^-|_A$ is upper bounded. Then any given\/
$\lambda\in\mathcal E^\sigma_g(A,1)$ is the\/ {\rm(}unique\/{\rm)} solution to
Problem\/~{\rm\ref{pr1}} if and only if
\begin{align}\label{desc1}g(\lambda+\vartheta)&\geqslant w\quad(\sigma-\lambda)\mbox{-a.e.},\\
g(\lambda+\vartheta)&\leqslant w\quad\lambda\mbox{-a.e.}\label{desc2}\end{align}
with some\/ $w\in\mathbb R$.\end{theorem}

\begin{proof} Since both $\lambda$ and $\vartheta$ are extendible, $g(\lambda+\vartheta)$ is finite n.e.\ on $D$, cf.\ Section~\ref{sec-gr}. Also note that any Borel set $Q\subset D$ with $c_g(Q)=0$ (equivalently, $c_\alpha(Q)=0$) cannot carry either of $\lambda$ and $\sigma$, the trace of $\sigma$ on any compact subset of $D$ being of finite $\alpha$-Riesz (equivalently, $\alpha$-Green) energy by (\ref{adm-c}).

For any $c\in\mathbb R$ write
\[A^+(c):=\bigl\{x\in A:\ g(\lambda+\vartheta)(x)>c\bigr\},\quad
A^-(c):=\bigl\{x\in A:\ g(\lambda+\vartheta)(x)<c\bigr\}.\]
Assume first that $\lambda$ solves Problem~\ref{pr1}. Then (\ref{desc1}) holds with $w:=L$, where
\[L:=\sup\,\bigl\{q\in\mathbb R: \ g(\lambda+\vartheta)\geqslant q\quad(\sigma-\lambda)\mbox{-a.e.}\bigr\}.\]
In turn, (\ref{desc1}) with $w=L$ implies $L<\infty$, because $g(\lambda+\vartheta)<\infty$ n.e.\ on
$D$, hence $(\sigma-\lambda)$-a.e.\ (see above), while $(\sigma-\lambda)(A)=\sigma(A)-1>0$. Also note that $L>-\infty$ since $g(\lambda+\vartheta)$
is lower bounded on $A$ by assumption.

We next proceed by establishing (\ref{desc2}) with $w:=L$. Suppose to the contrary that this fails, i.e.\ $\lambda(A^+(L))>0$. Since $g(\lambda+\vartheta)$
is Borel measurable, one can choose $c'\in(L,\infty)$ so that
$\lambda(A^+(c'))>0$. At the same time, as $c'>L$, it follows from the definition of $L$ that $(\sigma-\lambda)(A^-(c'))>0$. Therefore, there exist compact sets
$K_1\subset A^+(c')$ and $K_2\subset A^-(c')$ such that
\[0<\lambda(K_1)<(\sigma-\lambda)(K_2).\]
Write
\[\omega:=\lambda-\lambda|_{K_1}+b(\sigma-\lambda)|_{K_2},\text{ \ where \ }b:=\frac{\lambda(K_1)}{(\sigma-\lambda)(K_2)}\in(0,1).\]
Straightforward verification gives $\omega(A)=1$ and $\omega\leqslant\sigma$, and hence altogether
$\omega\in\mathcal E^\sigma_g(A,1)$. On the other hand,
\begin{align*}
&\langle\lambda+\vartheta,\omega-\lambda\rangle_g=\int\bigl[g(\lambda+\vartheta)-c'\bigr]\,d(\omega-\lambda)\\&{}=-\int\bigl[g(\lambda+\vartheta)-c'\bigr]\,d\lambda|_{K_1}
+b\int\bigl[g(\lambda+\vartheta)-c'\bigr]\,d(\sigma-\lambda)|_{K_2}<0,\end{align*}
which is impossible by Lemma~\ref{lequiv} applied to $\lambda$ and $\nu=\omega$. The contradiction obtained establishes
(\ref{desc2}).

Conversely, let (\ref{desc1}) and (\ref{desc2}) both hold with some $w\in\mathbb R$. Then
$\lambda(A^+(w))=0$ and
$(\sigma-\lambda)(A^-(w))=0$. For any $\nu\in\mathcal E^\sigma_g(A,1)$, we therefore have
\begin{align*}&\langle\lambda+\vartheta,\nu-\lambda\rangle_g=\int\bigl[g(\lambda+\vartheta)-w\bigr]\,d(\nu-\lambda)\\
&{}=\int\bigl[g(\lambda+\vartheta)-w\bigr]\,d\nu|_{A^+(w)}+\int\bigl[g(\lambda+\vartheta)-w\bigr]\,d(\nu-\sigma)|_{A^-(w)}\geqslant0,\end{align*}
which implies by Lemma~\ref{lequiv} that $\lambda$ solves, indeed,
Problem~\ref{pr1}.\end{proof}

\section{Auxiliary assertions}\label{sec-aux-as}

Throughout this section, $F=D^c$ is not $\alpha$-thin at infinity.

\begin{lemma}\label{laux1}For any\/ $\nu\in\mathcal E^\sigma_g(A,1)$,
\begin{align}\label{1}&\nu-\nu'\in\ddot{\mathcal E}_\alpha^\sigma(\mathbf A,\mathbf1),\\
\label{2}&\|\nu-\nu'\|^\cdot_\alpha=\|\nu\|_g,\\
\label{3}&\langle\nu-\nu',\vartheta-\vartheta'\rangle_\alpha^\cdot=\langle\nu,\vartheta\rangle_g,\\
\label{4}&\dot{G}_{\alpha,\vartheta-\vartheta'}(\nu-\nu')=G_{g,\vartheta}(\nu),\end{align}
and therefore
\begin{equation}\label{5}G_{g,\vartheta}^\sigma(A,1)\geqslant\dot{G}^\sigma_{\alpha,\vartheta-\vartheta'}(\mathbf A,\mathbf1).\end{equation}
\end{lemma}

\begin{proof} Theorems~\ref{bal-mass-th} and \ref{thm} applied to any given $\nu\in\mathcal E^\sigma_g(A,1)$ show that both $\nu-\nu'\in\dot{\mathcal E}_\alpha^\sigma(\mathbf A,\mathbf1)$ and (\ref{2}) hold. Also observe that in order to establish (\ref{1}), it is enough to construct a sequence $(\nu_k)\subset\mathcal E_\alpha\cap\breve{\mathfrak M}^+(A,1)$ with the properties
\begin{align}\label{ap1}&\nu_k\leqslant(1+k^{-1})\sigma\text{ \ for all\ }k,\\
&\lim_{k\to\infty}\,\|(\nu_k-\nu_k')-(\nu-\nu')\|^{\cdot}_\alpha=0,\label{ap2}\end{align}
because then $\nu-\nu'\in\dot{\mathcal E}_\alpha^\sigma(\mathbf A,\mathbf1)$ will be approximated in the weak energy norm on $\dot{\mathcal E}_\alpha$ by $\nu_k-\nu_k'\in\mathcal E_\alpha^{(1+\frac1k)\sigma}(\mathbf A,\mathbf1)$, exactly as required in definition (\ref{def_E2}).

Choose an increasing sequence $(K_k)$ of compact sets with the union $A$, and write
$\tilde{\nu}_k:=\nu|_{K_k}$. Since
$\tilde{\nu}_k(K_k)\uparrow\nu(A)=1$, there is no loss of generality in assuming
\begin{equation}\label{tilde}\frac{k}{k+1}\leqslant\tilde{\nu}_k(K_k)\leqslant1\text{ \ for all\ }k.\end{equation}
It follows from the definition of $\tilde{\nu}_k$ that
\begin{align}\label{appr1}&\tilde{\nu}_k\to\nu\text{ \ vaguely in $\mathfrak M(D)$},\\
\label{upar}&\lim_{k\to\infty}\,\|\tilde{\nu}_k\|_g=\|\nu\|_g,\end{align}
the non-trivial part of (\ref{upar}) being seen from (\ref{appr1}) in view of the lower semicontinuity of $E_g(\cdot)$ on $\mathfrak M^+(D)$.
Furthermore, for all $p\geqslant k$ we have $g\tilde{\nu}_k\leqslant g\tilde{\nu}_p$, hence
\[\langle\tilde{\nu}_k,\tilde{\nu}_p\rangle_g\geqslant\|\tilde{\nu}_k\|^2_g,\] and therefore
\[\|\tilde{\nu}_k-\tilde{\nu}_p\|^2_g\leqslant\|\tilde{\nu}_p\|^2_g-\|\tilde{\nu}_k\|^2_g,\]
which together with (\ref{upar}) proves that $(\tilde{\nu}_k)$ is a strong Cauchy
sequence in $\mathcal E_g$. Since the kernel $g$ is perfect, (\ref{appr1}) implies that $\tilde{\nu}_k\to\nu$ strongly in
$\mathcal E_g$. Having written
\begin{equation}\label{nyu}\nu_k:=\frac{\tilde{\nu}_k}{\tilde{\nu}_k(K_k)},\end{equation}
we obtain from this and (\ref{tilde})
\[\lim_{k\to\infty}\,\|\nu_k-\nu\|_g=0.\]
Applying Theorem~\ref{thm} to the bounded, hence extendible measure $\nu_k-\nu\in\mathcal E_g$ gives
\[0=\lim_{k\to\infty}\,\|\nu_k-\nu\|_g=\lim_{k\to\infty}\,\|(\nu_k-\nu)-(\nu_k'-\nu')\|^\cdot_\alpha,\]
which is (\ref{ap2}). As $S^{\nu_k}_D$ is compact, $E_g(\nu_k)<\infty$ implies $E_\alpha(\nu_k)<\infty$, cf.\ Lemma~\ref{eq-r-g}; hence, $(\nu_k)\subset\mathcal E_\alpha\cap\breve{\mathfrak M}^+(A,1)$. To complete the proof of (\ref{1}), it remains to observe that (\ref{ap1}) follows from $\tilde{\nu}_k\leqslant\nu\leqslant\sigma$, (\ref{tilde}), and (\ref{nyu}).

Finally, comparing
\[\|(\nu-\nu')+(\vartheta-\vartheta')\|^{\cdot2}_\alpha=\|\nu-\nu'\|^{\cdot2}_\alpha+\|\vartheta-\vartheta'\|^{\cdot2}_\alpha+
2\langle\nu-\nu',\vartheta-\vartheta'\rangle_\alpha^\cdot\]
with \[\|\nu+\vartheta\|^2_g=\|\nu\|^2_g+\|\vartheta\|^2_g+2\langle\nu,\vartheta\rangle_g\] and noting that, according to Theorem~\ref{thm},
\[\|(\nu-\nu')+(\vartheta-\vartheta')\|^\cdot_\alpha=\|\nu+\vartheta\|_g,\quad\|\nu-\nu'\|^\cdot_\alpha=\|\nu\|_g,\quad\|\vartheta-\vartheta'\|^\cdot_\alpha=\|\vartheta\|_g,\]
we obtain (\ref{3}). Combining (\ref{2}) and (\ref{3}) gives (\ref{4}).\end{proof}

\begin{lemma}\label{l-sep2}Assume that the separation condition\/ {\rm(\ref{dist})} holds. Then
\begin{equation}\label{eq-sep2}G^\sigma_{\alpha,\vartheta-\vartheta'}(\mathbf A,\mathbf1)=G^\sigma_{g,\vartheta}(A,1),\end{equation}
$G^\sigma_{\alpha,\vartheta-\vartheta'}(\mathbf A,\mathbf1)$ being defined by\/ {\rm(\ref{e-pr4})}. Moreover, a solution\/ $\lambda^\sigma_{\mathbf A}$ to Problem\/~{\rm\ref{pr4}} exists if and only if there is a solution\/ $\lambda^\sigma_A$ to Problem\/~{\rm\ref{pr1}}, and in the affirmative case
\begin{equation*}\lambda^\sigma_{\mathbf A}=\lambda^\sigma_A-(\lambda^\sigma_A)'.\end{equation*}
\end{lemma}

\begin{proof} As seen from (\ref{ext2}) and (\ref{5}), (\ref{eq-sep2}) will follow if we establish
\[G^\sigma_{\alpha,\vartheta-\vartheta'}(\mathbf A,\mathbf1)\geqslant G_{g,\vartheta}^\sigma(A,1).\]
For any fixed $\mu\in\mathcal E_\alpha^\sigma(\mathbf A,\mathbf1)$ we obtain from (\ref{proj}) and Corollary~\ref{cor}
\begin{equation}\label{6}\|\mu^+\|_g=\|\mu^+-(\mu^+)'\|_\alpha\leqslant\|\mu\|_\alpha,\end{equation}
where equality prevails in the inequality if and only if $\mu^-=(\mu^+)'$. Furthermore, in view of the absolute continuity of $\mu$,
\begin{equation}\label{7}\int g\vartheta\,d\mu^+=\int\kappa_\alpha(\vartheta-\vartheta')\,d\mu^+=
\int\kappa_\alpha(\vartheta-\vartheta')\,d\mu,\end{equation}
the former equality being valid by (\ref{3.4}) and the latter by $\kappa_\alpha\vartheta=\kappa_\alpha\vartheta'$ n.e.\ on $F$, cf.\ (\ref{bal-eq}). Since obviously $\mu^+\in\mathcal E_g^\sigma(A,1)$, we get from the last two displays
\begin{align}\label{8}G_{\alpha,\vartheta-\vartheta'}(\mu)&=\|\mu\|^2_\alpha+2\langle\mu,\vartheta-\vartheta'\rangle_\alpha\geqslant\|\mu^+\|_g^2+2\langle\mu^+,\vartheta\rangle_g\\
&{}=G_{g,\vartheta}(\mu^+)\geqslant G_{g,\vartheta}^\sigma(A,1),\notag\end{align}
which results in the required inequality by varying $\mu$ over $\mathcal E_\alpha^\sigma(\mathbf A,\mathbf1)$.

To establish the remaining assertion of the lemma, suppose first that $\lambda^\sigma_{\mathbf A}$ solves Problem~\ref{pr4}. Substituting $\lambda^\sigma_{\mathbf A}$ in place of $\mu$ into (\ref{8}) and then combining this with (\ref{eq-sep2}), we see that in fact equality prevails in either of the two inequalities. Problem~\ref{pr1} has therefore
a solution $\lambda^\sigma_A$; and moreover
\[(\lambda^\sigma_{\mathbf A})^+=\lambda^\sigma_A\text{ \ and \ }(\lambda^\sigma_{\mathbf A})^-=(\lambda^\sigma_A)'.\]
As for the last equality, see the discussion around (\ref{6}), referring to (\ref{proj}).

Conversely, suppose now that $\lambda^\sigma_A$ solves Problem~\ref{pr1}. Then $\mu_0:=\lambda^\sigma_A-(\lambda^\sigma_A)'$ belongs to $\mathcal E_\alpha^\sigma(\mathbf A,\mathbf1)$, which is seen by combining (\ref{1}) with (\ref{ext1}), and (\ref{8}) holds with equalities in place of both the inequalities. Thus $G_{\alpha,\vartheta-\vartheta'}(\mu_0)=G_{g,\vartheta}^\sigma(A,1)$, which together with (\ref{eq-sep2}) shows that $\mu_0$ thus defined solves Problem~\ref{pr4}.
\end{proof}

\begin{remark}The latter assertion in Lemma~\ref{l-sep2} fails in general once the separation condition (\ref{dist}) is dropped. For $\sigma=\infty$ and $\vartheta=0$, this can be seen from Remark~\ref{rem-g} and \cite[Example~10.1]{FZ-Pot2}, quoted in Section~\ref{sec:intr} above.\end{remark}

\begin{lemma}\label{l-ineq}If\/ {\rm(\ref{suff})} holds, then
\begin{equation*}\dot{G}^\sigma_{\alpha,\vartheta-\vartheta'}(\mathbf A,\mathbf1)=G^\sigma_{g,\vartheta}(A,1).\end{equation*}
\end{lemma}

\begin{proof} As seen from (\ref{5}), it is enough to establish
\[\dot{G}^\sigma_{\alpha,\vartheta-\vartheta'}(\mathbf A,\mathbf1)\geqslant G_{g,\vartheta}^\sigma(A,1).\]
Fix $\mu\in\ddot{\mathcal E}_\alpha^\sigma(\mathbf A,\mathbf1)$. By (\ref{def_E2}), for every $k\in\mathbb N$ there is $\mu_k\in\mathcal E_\alpha^{(1+\frac1k)\sigma}(\mathbf A,\mathbf1)$ with
\begin{equation*}\|\mu-\mu_k\|_\alpha^\cdot<k^{-1}.\end{equation*}
Applying the triangle inequality in $\dot{\mathcal E}_\alpha$, we then obtain from (\ref{dot0}) and (\ref{6})
\[\|\mu^+_k\|_g\leqslant\|\mu_k\|_\alpha=\|\mu_k\|^\cdot_\alpha=\|(\mu_k-\mu)+\mu\|^\cdot_\alpha
\leqslant\|\mu\|^\cdot_\alpha+k^{-1}.\]
Furthermore,
\[\langle\mu^+_k,\vartheta\rangle_g=\langle\mu_k,\vartheta-\vartheta'\rangle_\alpha=\langle\mu_k,\vartheta-\vartheta'\rangle_\alpha^\cdot=
\langle\mu_k-\mu,\vartheta-\vartheta'\rangle_\alpha^\cdot+\langle\mu,\vartheta-\vartheta'\rangle_\alpha^\cdot,\]
the first equality being valid by (\ref{7}) and the second by (\ref{dot0}). By the Cauchy--Schwarz inequality in $\dot{\mathcal E}_\alpha$,
\[\langle\mu_k-\mu,\vartheta-\vartheta'\rangle_\alpha^\cdot<k^{-1}\|\vartheta-\vartheta'\|_\alpha,\]
which together with the preceding two displays yields
\[G_{g,\vartheta}(\mu^+_k)=\|\mu^+_k\|^2_g+2\langle\mu^+_k,\vartheta\rangle\leqslant\|\mu\|^{\cdot2}_\alpha+
2\langle\mu,\vartheta-\vartheta'\rangle_\alpha^\cdot+\mathcal O(k^{-1}).\]
Noting that  $\mu^+_k\in\mathcal E_g^{(1+\frac1k)\sigma}(A,1)$, we therefore get
\[\dot{G}_{\alpha,\vartheta-\vartheta'}(\mu)\geqslant G_{g,\vartheta}^{(1+\frac1k)\sigma}(A,1)+\mathcal O(k^{-1}).\]
Since (\ref{suff}) holds, Theorem~\ref{th-g-cont} with $A_k:=A$ and $\sigma_k:=(1+k^{-1})\sigma$ can be applied, which gives by letting $k\to\infty$
\[\dot{G}_{\alpha,\vartheta-\vartheta'}(\mu)\geqslant G_{g,\vartheta}^\sigma(A,1).\]
The proof is completed by varying $\mu$ over $\ddot{\mathcal E}_\alpha^\sigma(\mathbf A,\mathbf1)$.\end{proof}

\section{Proofs of the assertions formulated in Section~\ref{sec-main}}\label{sec-proofs}

Throughout this section, $D^c$ is not $\alpha$-thin at infinity. The proofs presented below are based on auxiliary assertions established in Sections~\ref{sec-rel}, \ref{sec:aux}, and \ref{sec-aux-as} above.

\subsection{Proof of Theorem~\ref{main1}}\label{sec:pr1}Assume (\ref{suff}) holds. Then according to Theorem~\ref{th-g} there is the (unique) solution to Problem~\ref{pr1}, namely $\lambda^\sigma_A\in\mathcal E^\sigma_g(A,1)$ with
\[G_{g,\vartheta}(\lambda^\sigma_A)=G^\sigma_{g,\vartheta}(A,1).\]
Furthermore, we obtain from Lemma~\ref{laux1}, cf.\ (\ref{1}),
\[\dot{\lambda}^\sigma_{\mathbf A}:=\lambda^\sigma_A-(\lambda^\sigma_A)'\in\ddot{\mathcal E}_\alpha^\sigma(\mathbf A,\mathbf1),\]
which in view of (\ref{4}) gives
\[G^\sigma_{g,\vartheta}(A,1)=G_{g,\vartheta}(\lambda^\sigma_A)=\dot{G}_{\alpha,\vartheta-\vartheta'}(\dot{\lambda}^\sigma_{\mathbf A})
\geqslant\dot{G}^\sigma_{\alpha,\vartheta-\vartheta'}(\mathbf A,\mathbf1).\]
Applying Lemma~\ref{l-ineq} we see that equality in fact prevails in the inequality here, and hence $\dot{\lambda}^\sigma_{\mathbf A}$ thus defined solves, indeed, Problem~\ref{pr3}. On account of the uniqueness of a solution to Problem~\ref{pr3} (Lemma~\ref{l-unique}), this establishes Theorem~\ref{main1}.

\subsection{Proof of Theorem~\ref{infcap}} Assume that $c_g(A)=\infty$. Choose an increasing sequence $(K_k)$ of compact sets whose union equals $D$. As $c_g(K_k)<\infty$, it follows
from the subadditivity of $c_g(\cdot)$ on universally measurable sets \cite[Lemma~2.3.5]{F1} that
$c_g(A\setminus K_k)=\infty$ for all $k\in\mathbb N$. Hence, for every $k$ there is
$\nu_k\in\breve{\mathfrak M}^+(A\setminus K_k,1)$ with compact in $D$ support $S_D^{\nu_k}$ such that
\begin{equation}\label{to0}\|\nu_k\|^2_g\leqslant 1/k.\end{equation}
Clearly, the $K_k$ can be chosen successively so that $K_k\cup S^{\nu_k}_D\subset K_{k+1}$.
Any compact set $K\subset D$ is contained in a certain $K_k$ with $k$ large enough, and hence $K$ has
points in common with only finitely many $S^{\nu_k}_D$. Therefore $\sigma_0$ given by
\[\sigma_0(\varphi):=\sum_{k\in\mathbb N}\,\nu_k(\varphi)\text{ \ for all\ }\varphi\in C_0(D)\]
is a positive Radon measure on $D$ carried by $A$. Furthermore, $\sigma_0(A)=\infty$. For each
$\sigma\in\mathfrak C(A)\cup\{\infty\}$ such that $\sigma\geqslant\sigma_0$ we thus have
\[\nu_k\in\mathcal E^\sigma_g(A,1)\text{ \ for all\ }k\in\mathbb N,\]
which in view of (\ref{1}) implies
\begin{equation}\label{i}\nu_k-\nu_k'\in\ddot{\mathcal E}_\alpha^\sigma(\mathbf A,\mathbf1).\end{equation}
Assuming now in addition that $\vartheta=0$ (then $\dot{G}_{\alpha,\vartheta}(\mu)=\|\mu\|_\alpha^{\cdot2}$ for all $\mu\in\dot{\mathcal E}_\alpha$, hence $\dot{G}^\sigma_{\alpha,\vartheta}(\mathbf A,\mathbf1)\geqslant0$), we then obtain from (\ref{i}), (\ref{2}), and (\ref{to0})
\[0\leqslant\dot{G}^\sigma_{\alpha,\vartheta}(\mathbf A,\mathbf1)\leqslant\|\nu_k-\nu_k'\|_\alpha^{\cdot2}=\|\nu_k\|^2_g\leqslant 1/k\text{ \ for all\ }k\in\mathbb N,\]
and so $\dot{G}^\sigma_{\alpha,\vartheta}(\mathbf A,\mathbf1)=0$ for all $\sigma$ specified above. As $\|\cdot\|_\alpha^\cdot$ is a norm, $\dot{G}^\sigma_{\alpha,\vartheta}(\mathbf A,\mathbf1)$ cannot therefore be attained among the (nonzero) admissible measures $\mu\in\ddot{\mathcal E}_\alpha^\sigma(\mathbf A,\mathbf1)$.

\subsection{Proof of Corollary~\ref{lusin}} Let $\vartheta=0$. The implication (i)$\Rightarrow$(iii) follows from Theorem~\ref{main1}, (ii)$\Rightarrow$(i) from Theorem~\ref{infcap}, and (iii)$\Rightarrow$(ii) is obvious.

Assuming now $c_g(A)<\infty$, we proceed by proving (\ref{e-l}). According to Theorem~\ref{th-equi}, there is the $g$-equilibrium measure
$\gamma=\gamma_A$ on $A$. By Lusin's type theorem \cite[Theorem~3.6]{L} applied to each of
$\kappa_\alpha\gamma$ and $\kappa_\alpha\gamma'$, there exists for any $\varepsilon>0$ an open set
$\Omega\subset\mathbb R^n$ with $c_\alpha(\Omega)<\varepsilon$ such that $\kappa_\alpha\gamma$ and
$\kappa_\alpha\gamma'$ are both continuous relative to $\mathbb R^n\setminus\Omega$. On account of Remark~\ref{qe}, we see from (\ref{eq2}) and (\ref{3.4}) that there is no loss of generality in assuming $\kappa_\alpha\gamma=\kappa_\alpha\gamma'+1$ everywhere on $A\setminus\Omega$, which implies
$\kappa_\alpha\gamma=\kappa_\alpha\gamma'+1$ on $(\text{\rm Cl}_{\mathbb R^n}A)\setminus\Omega$.
As $\varepsilon$ is arbitrary, we thus have
\[\kappa_\alpha\gamma=\kappa_\alpha\gamma'+1\text{ \ q.e.\ on\ }\partial D\cap\text{\rm Cl}_{\mathbb R^n}A.\]
But, by (\ref{bal-eq}), $\kappa_\alpha\gamma=\kappa_\alpha\gamma'$ n.e.\ on
$\partial D$, hence q.e., for $\{\kappa_\alpha\gamma\ne\kappa_\alpha\gamma'\}$ is a
Borel set. In view of the preceding display, this is possible only provided that (\ref{e-l}) holds.

\subsection{Proof of Theorem~\ref{main2} and Corollary~\ref{cor-main2}}\label{sec-prec} Fix $\sigma\in\mathfrak C(A)$, and assume (\ref{suff}) holds. By Theorem~\ref{main1}, there is the (unique)  solution $\dot{\lambda}:=\dot{\lambda}^\sigma_{\mathbf A}$ to Problem~\ref{pr3}, which certainly satisfies (\ref{desc0a}) because of (\ref{th1-1}) and whose positive part $\dot{\lambda}^+$ solves Problem~\ref{pr1}. Assuming now in addition that $\kappa_\alpha\vartheta^-|_A$ is upper bounded, we then see from Theorem~\ref{desc-th2} that (\ref{desc1}) and (\ref{desc2}) both hold with $\lambda:=\dot{\lambda}^+$ and some $w\in\mathbb R$. But, on account of (\ref{3.4}), (\ref{def-f}), and (\ref{desc0a}),
\begin{equation*}g(\dot{\lambda}^++\vartheta)=\kappa_\alpha\dot{\lambda}+\kappa_\alpha(\vartheta-\vartheta')=\kappa_\alpha\dot{\lambda}+f\text{ \ n.e.\ on\ }D.\end{equation*}
Since any Borel subset of $D$ with $c_\alpha(\cdot)=0$ cannot carry either of $\dot{\lambda}^+$ and $\sigma$, substituting this into (\ref{desc1}) and (\ref{desc2}) gives  (\ref{desc1a}) and (\ref{desc2a}) with the same $w\in\mathbb R$.

In the rest of the proof, $\vartheta=0$. Then $\kappa_\alpha\dot{\lambda}+f=g\dot{\lambda}^+>0$ on $D$, which substituted into (\ref{desc2a}) implies $w>0$. Furthermore, (\ref{desc2a}) itself takes now the form
\[\kappa_\alpha\dot{\lambda}^+\leqslant w+\kappa_\alpha\dot{\lambda}^-\quad\dot{\lambda}^+\mbox{-a.e.}\]
Choose an increasing sequence $(K_k)$ of compact sets with the union $D$, and write $\dot{\lambda}^+_k:=\dot{\lambda}^+|_{K_k}$. Then the last display remains valid with $\dot{\lambda}^+$ replaced by  $\dot{\lambda}^+_k$, i.e.
\[\kappa_\alpha\dot{\lambda}^+_k\leqslant w+\kappa_\alpha\dot{\lambda}^-\quad\dot{\lambda}^+_k\mbox{-a.e.}\]
As seen from $E_g(\dot{\lambda}^+_k)\leqslant E_g(\dot{\lambda}^+)<\infty$ and Lemma~\ref{eq-r-g}, $E_\alpha(\dot{\lambda}^+_k)<\infty$. Therefore applying
\cite[Theorems~1.27, 1.29]{L} shows that the preceding display holds, in fact, everywhere on $\mathbb R^n$.
Having observed that $\kappa_\alpha\dot{\lambda}^+_k\uparrow\kappa_\alpha\dot{\lambda}^+$ pointwise on $\mathbb R^n$, we obtain (\ref{desc2''}) by letting $k\to\infty$. Combining (\ref{desc2''}) and (\ref{desc1a}) establishes (\ref{desc1''}).

Assuming now that $\alpha<2$ and $m(D^c)>0$, we proceed by proving (\ref{s1}). To the contrary, suppose there is $x_0\in S_D^{\sigma}$ with $x_0\not\in
S^{\dot{\lambda}^+}_D$. Then one can choose an open neighborhood $V$ of $x_0$ in $D$ so that $V\cap S^{\dot{\lambda}^+}_D=\varnothing$. Since $(\sigma-\dot{\lambda}^+)(V)>0$, it follows from (\ref{desc1''}) that there is $x_1\in V$ with the property
\begin{equation*}\kappa_\alpha\dot{\lambda}^+(x_1)=w+\kappa_\alpha(\dot{\lambda}^+)'(x_1).\end{equation*}
As $\kappa_\alpha\dot{\lambda}^+$ is $\alpha$-harmonic on $V$, while
$w+\kappa_\alpha(\dot{\lambda}^+)'$ $\alpha$-super\-harmonic on $\mathbb R^n$ \cite[Chapter~I, Section~6, n$^\circ$~20]{L}, we see from this and (\ref{desc2''}) by \cite[Theorem~1.28]{L} that
\[\kappa_\alpha\dot{\lambda}=w\quad m\mbox{-a.e.\ on\ }\mathbb R^n.\]
This implies $w=0$, because $\kappa_\alpha\dot{\lambda}=0$ n.e., hence $m$-a.e., on $D^c$. A
contradiction.

\subsection{Proof of Theorem~\ref{main2-sep}}\label{sec:prlast} Assuming that the separation condition (\ref{dist}) holds, fix $\sigma\in\mathfrak C(A)$ and $\dot{\lambda}\in\ddot{\mathcal E}_\alpha^\sigma(\mathbf A,\mathbf1)$. Combining Corollary~\ref{cor-sep} with Lemma~\ref{l-sep2} implies that this $\dot{\lambda}$ solves Problem~\ref{pr3} if and only if both (a) and (b) are valid, where
\begin{itemize}\item[(a)] $\dot{\lambda}^+$ solves Problem~\ref{pr1},
\item[(b)] (\ref{desc0a}) holds, i.e. $\dot{\lambda}^-=(\dot{\lambda}^+)'$.\end{itemize}
Assuming now in addition that $\kappa_\alpha\vartheta^-|_A$ is upper bounded, we see by Theorem~\ref{desc-th2} that (a) is equivalent to the claim that $\dot{\lambda}^+$ satisfies both (\ref{desc1}) and (\ref{desc2}) with some $w\in\mathbb R$. Since $g(\dot{\lambda}^++\vartheta)=\kappa_\alpha\dot{\lambda}+f$ n.e.\ on $D$ by (b), this claim in turn is equivalent to the assertion that $\dot{\lambda}$ satisfies both (\ref{desc1a}) and (\ref{desc2a}) with the same $w\in\mathbb R$ (cf.\ the first paragraph in Section~\ref{sec-prec}). This altogether establishes Theorem~\ref{main2-sep}.

\subsection{Proof of Theorems~\ref{th-al-cont} and \ref{th-al-cont'}} Let $(A_k)$ and $(\sigma_k)$ be as described in the beginning of Section~\ref{sec-cont}, and assume that (\ref{c1'}) holds. Write $\mathbf A_k:=(A_k,F)$. By Lemma~\ref{l-ineq} applied to $\mathbf A_k$ and $\sigma_k$, resp.\ $\mathbf A$ and $\sigma$,
\begin{align*}\dot{G}_{\alpha,\vartheta-\vartheta'}^{\sigma_k}(\mathbf A_k,\mathbf1)&=G^{\sigma_k}_{g,\vartheta}(A_k,1)\text{ \ for all\ }k,\\
\dot{G}_{\alpha,\vartheta-\vartheta'}^\sigma(\mathbf A,\mathbf1)&=G^\sigma_{g,\vartheta}(A,1),\end{align*}
which combined with (\ref{g-cont1}) results in (\ref{g-cont1'}).

According to Theorem~\ref{main1}, there exists the (unique) solution $\dot{\lambda}_k:=\dot{\lambda}^{\sigma_k}_{\mathbf A_k}$, resp.\ $\dot{\lambda}:=\dot{\lambda}^\sigma_{\mathbf A}$, to Problem~\ref{pr3} with $\mathbf A_k$ and $\sigma_k$, resp.\ $\mathbf A$ and $\sigma$, and moreover
\[\dot{\lambda}_k=\dot{\lambda}_k^+-(\dot{\lambda}_k^+)',\quad\dot{\lambda}=\dot{\lambda}^+-(\dot{\lambda}^+)',\]
where $\dot{\lambda}_k^+$, resp.\ $\dot{\lambda}^+$, solves Problem~\ref{pr1} with $A_k$ and $\sigma_k$, resp.\ $A$ and $\sigma$. Applying Theorem~\ref{thm} to the (bounded, hence extendible) measure $\dot{\lambda}_k^+-\dot{\lambda}^+\in\mathcal E_g$, we get
\[\|\dot{\lambda}_k^+-\dot{\lambda}^+\|_g=\|(\dot{\lambda}_k^+-\dot{\lambda}^+)-(\dot{\lambda}_k^+-\dot{\lambda}^+)'\|^\cdot_\alpha,\]
with together with the preceding display and (\ref{g-cont2}) gives (\ref{g-cont2'}).

Assume now in addition that (\ref{sep}) holds; then each of $\mathbf A_k$ and $\mathbf A$ has separated plates. Therefore, by Lemma~\ref{l-sep} and Corollary~\ref{cor-sep}, $\lambda_k:=\dot{\lambda}_k$, resp.\ $\lambda:=\dot{\lambda}$, belongs to $\mathcal E^\sigma_\alpha(\mathbf A_k,\mathbf1)$, resp.\ $\mathcal E^\sigma_\alpha(\mathbf A,\mathbf1)$, and moreover, solves Problem~\ref{pr4} with $\mathbf A_k$ and $\sigma_k$, resp.\ $\mathbf A$ and $\sigma$. Combining (\ref{g-cont1'}) with (\ref{ext2}), applied to $\mathbf A_k$ and $\sigma_k$, resp.\ $\mathbf A$ and $\sigma$, gives (\ref{g-cont1''}), while (\ref{g-cont2'}) in view of (\ref{dot0}) can now be rewritten as (\ref{g-cont2''}).

Since all the $\lambda_k,\lambda$ belong to $\mathcal E^{\leqslant q}_\alpha(\mathbf A_1)$ with $q\geqslant2$ (for the notation, see Section~\ref{sec:intr}), (\ref{g-cont2''}) means that $\lambda_k\to\lambda$ in the strong topology on $\mathcal E^{\leqslant q}_\alpha(\mathbf A_1)$, determined by the standard energy norm. But according to \cite[Theorem~1]{Z1} quoted in Section~\ref{sec:intr} (cf.\ also footnote~\ref{foot3}), the strong topology on $\mathcal E^{\leqslant q}_\alpha(\mathbf A_1)$ is stronger than the vague topology, the condenser $\mathbf A_1$ satisfying the separation condition (\ref{sep}). Thus
\[\lambda_k\to\lambda\text{ \ vaguely in\ }\mathfrak M(\mathbb R^n),\]
which is however equivalent to (\ref{g-cont3'}), the equivalence being seen from (\ref{sep}) by applying the Tietze--Urysohn extension theorem \cite[Theorem~0.2.13]{E}.

\section{An advantage of weak energy for condenser problems}\label{sec-adv}

In this section $\sigma=\infty$, $\vartheta=0$, $c_g(A)<\infty$, and $D^c$ is not $\alpha$-thin at infinity. Then Problem~\ref{pr3} reduces to the problem on the existence of $\dot{\lambda}_{\mathbf A}\in\ddot{\mathcal E}_\alpha(\mathbf A,\mathbf1)$, cf.\ (\ref{circ}), with
\begin{equation}\label{wdot}\|\dot{\lambda}_{\mathbf A}\|^{\cdot2}_\alpha=\dot{w}_\alpha(\mathbf A,\mathbf1):=\inf_{\mu\in\ddot{\mathcal E}_\alpha(\mathbf A,\mathbf1)}\|\mu\|^{\cdot2}_\alpha.\end{equation}
According to Theorem~\ref{main1}, this problem has the (unique) solution $\dot{\lambda}_{\mathbf A}$, and moreover
\begin{equation}\label{sol-inf}\dot{\lambda}_{\mathbf A}=(\gamma_A-\gamma_A')/c_g(A),\end{equation}
where $\gamma_A$ is the $g$-equilibrium measure on $A$ (cf.\ Remark~\ref{rem-g}).

Having now replaced in our definitions weak energy by standard energy, we obtain the class $\mathcal E_\alpha(\mathbf A,\mathbf1)$ and the extremal value $w_\alpha(\mathbf A,\mathbf1)$ (cf.\ (\ref{eprcl}) and (\ref{epr})) in place of $\ddot{\mathcal E}_\alpha(\mathbf A,\mathbf1)$ and $\dot{w}_\alpha(\mathbf A,\mathbf1)$, respectively.
To justify an advantage of weak energy for condenser problems, we shall show that
\begin{itemize}
\item[($\mathcal W$)]{\it $w_\alpha(\mathbf A,\mathbf1)$ cannot in general be attained within the class\/ $\mathcal E_\alpha(\mathbf A,\mathbf1)$},
\end{itemize}
and hence the quoted assertion on the solvability of  problem (\ref{wdot}) fails in general once the problem is reformulated in the setting of standard energy. (Observe that according to  Corollary~\ref{cor-sep}, this may occur only provided that $\text{\rm dist}(A,F)=0$.)

\begin{lemma}\label{l-f} $\dot{w}_\alpha(\mathbf A,\mathbf1)=w_\alpha(\mathbf A,\mathbf1)$.\end{lemma}

\begin{proof} It is seen from (\ref{circ}) that $\mathcal E_\alpha(\mathbf A,\mathbf1)\subset\ddot{\mathcal E}_\alpha(\mathbf A,\mathbf1)$, which together with (\ref{dot0})
yields $\dot{w}_\alpha(\mathbf A,\mathbf1)\leqslant w_\alpha(\mathbf A,\mathbf1)$. For the opposite, choose $(\mu_k)\subset\ddot{\mathcal E}_\alpha(\mathbf A,\mathbf1)$ so that
\[\lim_{k\to\infty}\|\mu_k\|^{\cdot2}_\alpha=\dot{w}_\alpha(\mathbf A,\mathbf1),\]
and for every $\mu_k$ choose $\tilde{\mu}_k\in\mathcal E_\alpha(\mathbf A,\mathbf1)$ with $\|\mu_k-\tilde{\mu}_k\|_\alpha^\cdot<k^{-1}$; this $\tilde{\mu}_k$ exists by (\ref{circ}). Thus
\[w_\alpha(\mathbf A,\mathbf1)\leqslant\|\tilde{\mu}_k\|^2_\alpha=\|\tilde{\mu}_k\|^{\cdot2}_\alpha=\|(\tilde{\mu}_k-\mu_k)+\mu_k\|^{\cdot2}_\alpha\leqslant(\|\mu_k\|_\alpha^\cdot+k^{-1})^2,\]
and letting here $k\to\infty$ completes the proof.\end{proof}

For $n=3$ and $\alpha=2$, choose $D$ and $A$ as in \cite[Example~10.1]{FZ-Pot2}; then $c_{g^2_D}(A)<\infty$ and moreover $E_2(\gamma_A)=\infty$, where $\gamma_A$ is the $2$-Green equilibrium measure on $A$. Suppose, contrary to ($\mathcal W$), that there is
$\lambda_{\mathbf A}\in\mathcal E_2(\mathbf A,\mathbf1)$ with $\|\lambda_{\mathbf A}\|^2_2=w_2(\mathbf A,\mathbf1)$. Since $c_{g^2_D}(A)<\infty$, there is also $\dot{\lambda}_{\mathbf A}\in\ddot{\mathcal E}_2(\mathbf A,\mathbf1)$ with $\|\dot{\lambda}_{\mathbf A}\|^{\cdot2}_2=\dot{w}_2(\mathbf A,\mathbf1)$. As $\lambda_{\mathbf A}\in\ddot{\mathcal E}_2(\mathbf A,\mathbf1)$ and $\|\lambda_{\mathbf A}\|^2_2=\|\lambda_{\mathbf A}\|^{\cdot2}_2$, Lemma~\ref{l-f} implies that $\lambda_{\mathbf A}$ along with $\dot{\lambda}_{\mathbf A}$ solves problem (\ref{wdot}), and hence necessarily (see Lemma~\ref{l-unique} and (\ref{sol-inf}))
\[\lambda^+_{\mathbf A}=\gamma_A/c_{g^2_D}(A).\]
This is however impossible because $E_2(\gamma_A)=\infty$.

\section{Acknowledgement} I thank Bent Fuglede for many fruitful discussions on the topic of the paper.


\begin{thebibliography}{99}

\bibitem{Bec} Beckermann, B., Gryson, A.: Extremal rational functions on symmetric discrete
sets and superlinear convergence of the ADI method. Constr. Approx. {\bf 32}, 393--428 (2010)

\bibitem{B2} Bourbaki, N.: Elements of Mathematics. Integration. Chapters~1--6.
Springer, Berlin (2004)

\bibitem{Brelot} Brelot, M.: Sur le r\^{o}le du point \`{a} l'infini dans la th\'{e}orie des fonctions harmoniques. Ann. \'{E}c. Norm. Sup., {\bf 61}, 301--332 (1944)

\bibitem{Brelo2} Brelot, M.: On Topologies and Boundaries in Potential Theory. Lecture Notes in Math.,
vol.~175. Springer,
Berlin (1971)

\bibitem{Ca1} Cartan, H.: Th\'eorie du potentiel Newtonien:
\'energie, capacit\'e, suites de potentiels. Bull. Soc. Math. Fr.
{\bf 73}, 74--106 (1945)

\bibitem{Ca2} Cartan, H.: Th\'eorie g\'en\'erale du balayage en potentiel newtonien. Ann. Univ. Grenoble
{\bf 22}, 221--280 (1946)

\bibitem{D1} Deny, J.: Les potentiels d'\'energie finie. Acta Math.\ {\bf 82}, 107--183 (1950)

\bibitem{Doob} Doob, J.L.: Classical Potential Theory and Its Probabilistic Counterpart. Springer,
Berlin (1984)

\bibitem{DFHSZ} Dragnev, P.D., Fuglede, B., Hardin, D.P., Saff, E.B., Zorii, N.: Minimum Riesz energy
problems for a condenser with touching plates. Potential Anal.\ {\bf 44}, 543--577 (2016)

\bibitem{DFHSZ2} Dragnev, P.D., Fuglede, B., Hardin, D.P., Saff, E.B., Zorii, N.: Condensers with touching
plates and constrained minimum Riesz and Green energy problems. Constr. Approx. {\bf 50}, 369--401 (2019)

\bibitem{DFHSZ1} Dragnev, P.D., Fuglede, B., Hardin, D.P., Saff, E.B., Zorii, N.: Constrained minimum Riesz
energy problems for a condenser with intersecting plates. J. Anal. Math., to appear.
ArXiv:1710.01950v2 (2018)

\bibitem{DS} Dragnev, P.D., Saff, E.B.: Constrained energy problems with applications
to orthogonal polynomials of a discrete variable. J.~Anal.~Math.
{\bf 72}, 223–-259 (1997)

\bibitem{E}
Edwards, R.E.: Functional Analysis. Theory and Applications. Holt,
Rinehart and Winston, New York (1965)

\bibitem{F1} Fuglede, B.: On the theory of potentials in
locally compact spaces. Acta Math.\ {\bf 103}, 139--215 (1960)

\bibitem{FZ} Fuglede, B., Zorii, N.: Green kernels associated with Riesz kernels. Ann. Acad. Sci. Fenn. Math. {\bf 43}, 121--145 (2018)

\bibitem{FZ-Pot1} Fuglede, B., Zorii, N.: An alternative concept of Riesz energy of measures with application to generalized condensers. Potential Anal. {\bf 51}, 197--217 (2019)

\bibitem{FZ-Pot2} Fuglede, B., Zorii, N.: Various concepts of Riesz energy of measures and application to condensers with touching plates. Potential Anal. (2019) https://doi.org/10.1007/s11118-019-09803-w

\bibitem{Gauss} Gauss, C.F.: Allgemeine Lehrs\"atze in Beziehung auf die im
verkehrten Verh\"altnisse des Quadrats der Entfernung wirkenden
Anziehungs-- und Absto{\ss}ungs--Kr\"afte (1839). Werke {\bf 5}, 197--244 (1867)

\bibitem{HWZ}
Harbrecht, H., Wendland, W.L., Zorii, N.: Riesz minimal
energy problems on $C^{k-1,k}$-man\-if\-olds. Math. Nachr. {\bf 287}, 48--69 (2014)

\bibitem{L} Landkof, N.S.: Foundations of Modern Potential Theory. Springer, Berlin (1972)

\bibitem{OWZ}
Of, G., Wendland, W.L., Zorii, N.: On the numerical solution of minimal energy problems. Complex Var.
Elliptic Equ. {\bf 55}, 991--1012 (2010)

\bibitem{O}
Ohtsuka, M.: On potentials in locally compact spaces.
J. Sci. Hiroshima Univ. Ser.~A-1 {\bf 25}, 135--352 (1961)

\bibitem{R} Rakhmanov, E.A.: Equilibrium
measure and the distribution of zeros of extremal polynomials of a
discrete variable. Sb. Math. {\bf 187}, 1213--1228 (1996)

\bibitem{Riesz} Riesz, M.: Int\'egrales de Riemann--Liouville et potentiels. Acta Szeged {\bf 9}, 1--42 (1938)

\bibitem{ST} Saff, E.B., Totik, V.: Logarithmic potentials with external fields. Sprin\-ger, Berlin (1997)

\bibitem{Z1} Zorii, N.: A problem of minimum energy for space
condensers and Riesz kernels.  Ukr. Math. J. \textbf{41}, 29--36
(1989)

\bibitem{ZPot1} Zorii, N.: Interior capacities of condensers in locally compact spaces. Potential Anal.
{\bf 35}, 103--143 (2011)

\bibitem{ZPot2} Zorii, N.: Equilibrium problems for infinite dimensional vector potentials with
external fields. Potential Anal. {\bf 38}, 397--432 (2013)

\bibitem{ZPot3} Zorii, N.: Necessary and sufficient conditions for the solvability of
the Gauss variational problem for infinite dimensional vector measures. Potential Anal. {\bf 41}, 81--115 (2014)

\bibitem{Z-bal} Zorii, N.: A theory of inner Riesz balayage and its applications. ArXiv:1910.09946 (2019)
\end{thebibliography}
\end{document}